\documentclass[11pt]{amsart}

\usepackage{graphicx, amsthm, amsfonts, amscd,epstopdf} 
\usepackage[all]{xy}
\newcommand{\e}{\ensuremath{\mathbf{e}}}
\newcommand{\x}{\ensuremath{\mathbf{x}}}
\newcommand{\y}{\ensuremath{\mathbf{y}}}

\newcommand{\pd}[2]{\ensuremath{\frac{\partial #1}{\partial #2}}}
\newcommand{\proj}{\ensuremath{\pi}}

\DeclareMathOperator{\ind}{Ind}

\DeclareMathOperator{\writhe}{wr}

\newcommand{\rr}{\ensuremath{\mathbb{R}}}
\newcommand{\zz}{\ensuremath{\mathbb{Z}}}

\theoremstyle{plain}
\newtheorem{thm}{Theorem}[section]

\newtheorem{cor}[thm]{Corollary}
\newtheorem{lem}[thm]{Lemma}

\newtheorem{prop}[thm]{Proposition}

\theoremstyle{definition}
\newtheorem{defn}[thm]{Definition}

\theoremstyle{remark}
\newtheorem{rem}[thm]{Remark}
\newtheorem{exam}[thm]{Example}

\numberwithin{equation}{section}

  \def\bot{\beta}
\def\top{\tau}

\def\x{\mathbf{x}}
\def\y{\mathbf{y}}

\begin{document}

\title[Existence and Squeezing of Lagrangian Cobordisms]
{Obstructions to the Existence and Squeezing of Lagrangian
Cobordisms}

\author[J. Sabloff]{Joshua M. Sabloff} \address{Haverford College,
  Haverford, PA 19041} \email{jsabloff@haverford.edu}

\author[L. Traynor]{Lisa Traynor} \address{Bryn Mawr College, Bryn
  Mawr, PA 19010} \email{ltraynor@brynmawr.edu}

\begin{abstract}
  Capacities that provide both qualitative and quantitative
  obstructions to the existence of a Lagrangian cobordism between two
  $(n-1)$-dimensional submanifolds in parallel hyperplanes of
  $\mathbb{R}^{2n}$ are defined using the theory of generating
  families.  Qualitatively, these capacities show that, for example,
  in $\mathbb R^4$ there is no Lagrangian cobordism between two
  $\infty$-shaped curves with a negative crossing when the lower end
  is ``smaller''.  Quantitatively, when the boundary of a Lagrangian
  ball lies in a hyperplane of $\mathbb{R}^{2n}$, the capacity of the
  boundary gives a restriction on the size of a rectangular cylinder
  into which the Lagrangian ball can be squeezed.
\end{abstract}

\maketitle

% ****************************************

% ********************
\section{Introduction}
\label{sec:intro}

A fundamental problem in symplectic topology is to understand the
boundary between flexibility (when symplectic objects behave like
topological objects) and rigidity (when symplectic behavior is more
restrictive).  In this paper, we investigate  flexibility and
rigidity in the setting of existence questions for Lagrangian
cobordisms in the standard symplectic $\rr^{2n}$.

% **********
\subsection{Questions and Results}
\label{ssec:q-and-r}

Consider $\rr^{2n} = T^*(\rr^{n})$ with coordinates $(x_1,\ldots
x_{n}, y_1, \ldots y_{n})$ and the standard symplectic form $\omega =
\sum dx_i \wedge dy_i$.   Consider an $(n-1)$-dimensional submanifold $L_a$ in $\{
y_n = a \}$ and an $(n-1)$-dimensional submanifold $L_b$ in $\{ y_n =
b\}$.  A natural and important question to ask is whether or not there
is a Lagrangian cobordism between these submanifolds; that is, does
there exist an $n$-dimensional submanifold $L_{[a,b]}$ of $\rr^{2n}$ with $\omega
|_{TL_{[a,b]}} = 0$ that intersects $\{y_n=a\}$ transversally to form $L_a$
and $\{y_n = b\}$ transversally to form $L_b$?   

Topological data provides restictions on such a cobordism.  We will first focus
on $\rr^4$, though the restrictions and our results will
have higher dimensional analogues.  Let 
$\pi: \rr^{4} \to \rr^{2}$ be the
projection to the $(x_1, y_1)$ coordinates.   
It is easy to see that a cobordism between two connected curves $L_a$ and
$L_b$ can exist only if $\pi (L_a)$ and $\pi(L_b)$ bound the same
signed area and have the same winding number. Further, if
$\writhe(L_a)$ denotes the writhe of the diagram $\pi(L_a)$ with
respect to the blackboad framing, the canonical isomorphism between
the tangent and normal bundles of a Lagrangian submanifold then leads
to the following restriction on the Euler characteristic of the
cobordism:
\begin{equation} \label{eqn:euler}
  \writhe(L_b) - \writhe(L_a) = \chi(L_{[a,b]}).
\end{equation}

In this paper, we produce obstructions that
go beyond these basic ones. These new obstructions can be viewed both
qualitatively and quantitatively.  For an example of the former,
consider the following $\infty$-shaped curves in $\rr^3$:
\begin{equation*}
  8^{1}_\pm (r)
  := \{(x_1, x_2, y_1):  
  x_1^2 + x_2^2 = r^2, \quad y_1 = \pm 2x_1x_2\}.
\end{equation*}
Let $i_a$ denote the inclusion of $\rr^3$ into $\{ y_2 = a
\}$.  On one hand, it is possible to construct a Lagrangian cobordism
between inclusions of $8^1_-(R)$ and $8^1_-(r)$ when the 
larger curve is included at a lower level:  for example,
as explicitly shown in Example~\ref{exam:cylinder_construct},
 there exists a Lagrangian submanifold in $\rr^4$ so that:
$$\left( L \cap \{ y_2 = 1\} \right) = i_{1} (8^1_-(5)), \qquad 
\left( L \cap \{ y_2 = 10\} \right)
= i_{10} (8^1_-(4)).$$ On the other hand, a basic question probing the
rigidity of Lagrangian cobordisms, posed by Y.\ Eliashberg, is whether
or not there exists a Lagrangian cobordism between $\infty$-shaped
curves with a negative crossing when the bottom end is smaller.  This
situation is summarized graphically in Figure~\ref{fig:basic-slices}.

\begin{figure}   
  \centerline{\includegraphics{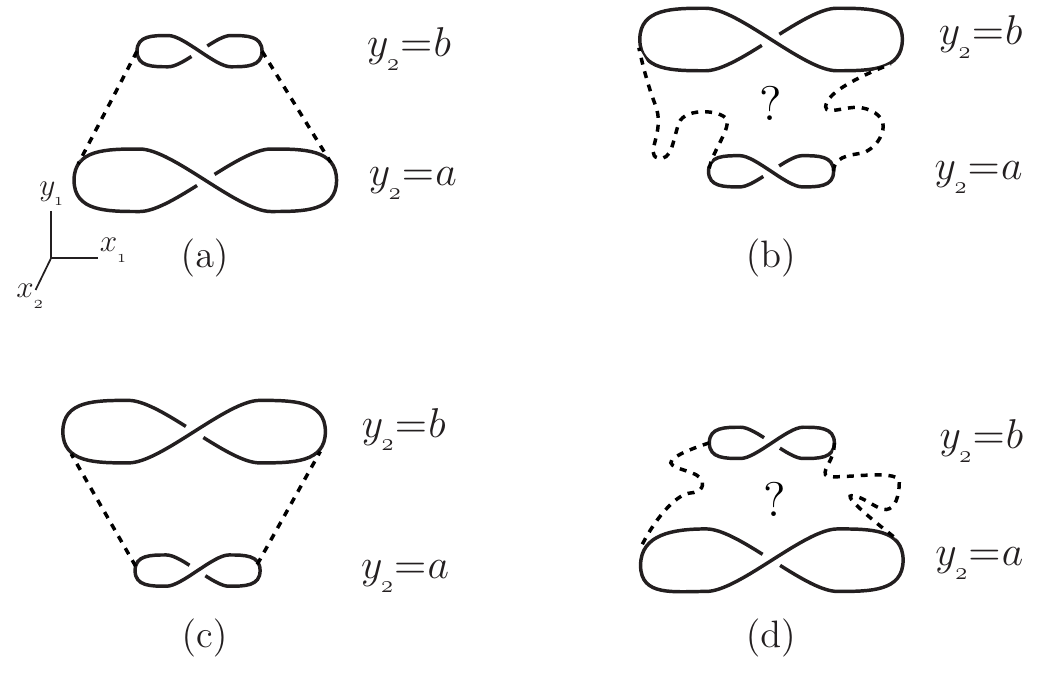}}
  \caption{ Parts (a) and (c) show two pairs of curves that can be
    joined by a Lagrangian cobordism.  Is there a Lagrangian cobordism
    as in parts (b) and (d) when the ordering of the heights of these
    curves is changed? }
  \label{fig:basic-slices}
\end{figure}

The following theorem gives an answer to this question:

\begin{thm}  
  \label{thm:8cobord} 
  If $r \leq R$ and $a < b$, there does not exist a Lagrangian
  cobordism $L \subset \{ a \leq y_2 \leq b \} \subset \rr^4$ 
%  that is  transverse to $\{y_2 = a\}$ and $\{ y_2 = b \}$ and satisfies
  with 
  $$(L \cap \{ y_2 = a\} )= i_{a}(8^1_-(r) ),
  \qquad (L \cap \{ y_2 = b\} )= i_{b}(8^1_-(R) ).$$
\end{thm}
\noindent
There are analogous statements about cobordisms between
$\infty$-shaped curves with positive crossings: in this case, there
does exist a Lagrangian cobordism when the bottom curve is smaller but
not if the bottom curve is larger.  This theorem illustrates an
asymmetry in some Lagrangian cobordisms.  There are many pairs of
curves besides the $\infty$-shaped curves that can be realized as the
ends of a Lagrangian cobordism in $\rr^4$ but not if the
ordering of their heights is changed.  In particular, the asymmetry of
Lagrangian cobordisms between curves extends to any Lagrangian
cobordism between connected curves that can be extended to a
planar Lagrangian that agrees with $\{ y_1 = y_2 = 0\}$ outside
a compact set.  This is proved in
\cite{arnold-paper}, using the techniques of this paper, where it
is shown that the set of connected negative (or
positive) hyperplane slices of flat-at-infinity planar Lagrangians in $\rr^4$ has
the structure of a partially ordered set.

On the quantitative side, we show that it is possible to measure the
``size" of a Lagrangian disk with boundary in a hyperplane.  
Consider a Lagrangian disk $L^2 \subset \rr^4$ that
transversally intersects $\{ y_2 = a \}$ at its boundary 
  $\partial L = i_a(8^1_-(r))$.  The projection
of this boundary curve to the $x_1y_1$-plane has two ``lobes'' having
equals areas (of opposite signs); denote the absolute value of the
area of one of these lobes by $A$.  In fact, the lobe area $A$ determines
whether the Lagrangian disk can be squeezed into a  rectangular
cylinder $\rr^2 \times R$, where $R$ is a rectangle in the
$x_2y_2$-coordinates; see Figure~\ref{fig:non-squeeze}.  More
precisely, we have:

\begin{thm}
  \label{thm:8squeeze}
  Suppose a Lagrangian disk $L \subset \rr^4$ transversally intersects
  $\{ y_2 = a \}$ and   
  $(L \cap \{y_2=a\} ) = \partial L =  i_a(8^1_\pm(r))$, for some $r > 0$. 
  %Suppose a Lagrangian disk $L \subset \rr^4$ has boundary $\partial L
  %= (L \cap \{y_2=a\} ) = i_a(8^1_\pm(r))$, for some $r > 0$.  
     Suppose that each lobe of
  $\pi(\partial L)$ bounds a region with area of absolute value $A$.
  If the entire Lagrangian disk lies in the rectangular cylinder
  $$C = \left\{ (x_1,x_2,y_1,y_2) \;:\; (x_2, y_2) \in I_{x_2} \times I_{y_2} \right\},$$ 
  for some intervals $I_{x_2} \subset \rr$ of length $\ell$ and $I_{y_2}
  \subset \rr$ of length $w$, then:
  $$\ell \cdot w \geq A.$$
\end{thm}

\begin{figure}
  \centerline{\includegraphics[width=5in]{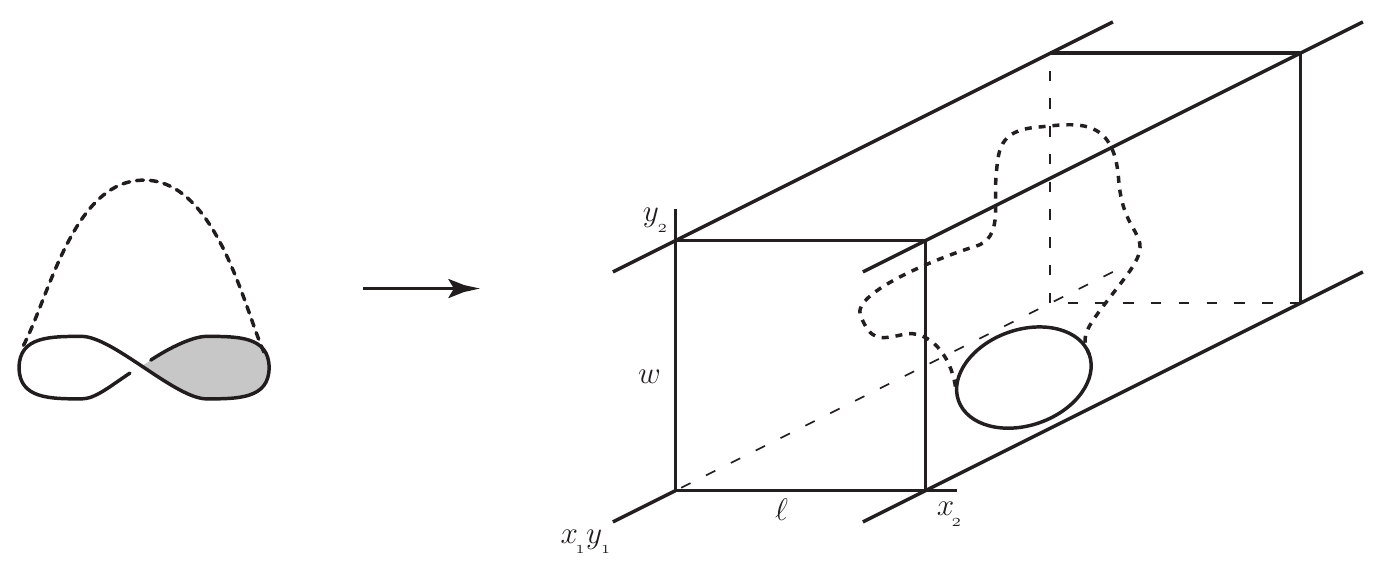}}
  \caption{If the shaded area on the left is greater than the area $l
    \cdot w$, then the Lagrangian disk on the left cannot be squeezed
    into the rectangular cylinder on the right.}
  \label{fig:non-squeeze}
\end{figure}

To prove and generalize the above two theorems, we pass 
to Lagrangians without boundary, namely ``unknotted
planar Lagrangians.''  Let $L_0$ denote the zero-section of
$T^*\rr^n$.  We say that a Lagrangian submanifold is {\bf planar} if
it is diffeomorphic to $\rr^n$; a planar Lagrangian is {\bf
  flat-at-infinity} if it agrees with $L_0$ outside a compact set of
$\rr^{2n}$, and is {\bf unknotted} if it is Lagrangian isotopic to
$L_0$ via a compactly supported symplectic isotopy of $\rr^{2n}$.  An
unknotted planar Lagrangian will be flat-at-infinity; in $\rr^4$, a
flat-at-infinity planar Lagrangian is always unknotted
\cite{ep:local-knots}.  This together with a gluing result allows us to
avoid in the statements of Theorems~\ref{thm:8cobord} 
 and \ref{thm:8squeeze} an ``extendability"
hypothesis  that will appear
in later theorems.   As shall be described in
Section~\ref{sec:background}, unknotted planar Lagrangians can be
studied using the technique of generating families.

We now set some terminology.  We say that $a$ is a {\bf generic
  height} for an unknotted planar Lagrangian $L$ if $L$ is transverse
to the hyperplane $\{y_n = a\}$.  For a generic height $a$, the
\textbf{slice $L_a$} is the intersection $L \cap \{ y_n = a \}$. Such
a slice is a submanifold of $\{y_n = a\} \simeq \rr^{2n-1}$; the
exactness of $L$ easily implies that the projection $\pi(L_a)$ of
$L_a$ to $\rr^{2n-2} = \{(x_1, \ldots, x_{n-1}, y_1, \ldots,
y_{n-1})\}$ is an exact Lagrangian immersion and, further, that the
projection has trivial Maslov class.  This generalizes the fact that
in $\rr^4$, the projection $\pi(L_a)$ bounds zero signed area and has
winding number zero.

In higher dimensions, Theorem~\ref{thm:8squeeze} generalizes
to show the nonexistence of cobordisms that can be extended to  
unknotted planar Lagrangians.
 First, the curves $8^1_\pm(r)$ should be generalized to:
\begin{align*}
  8^{n-1}_\pm (r)
  := \{(&x_1, \dots,  x_n, y_1, \dots, y_{n-1}):  \\
  &\sum_{i=1}^n x_i^2 = r^2, \quad y_i = \pm 2x_ix_{n}, \quad i= 1,
  \dots, n-1 \}.
\end{align*}
A construction similar to that in $\rr^4$ provides an unknotted planar
Lagrangian whose restriction gives a
cobordism from $8^{n-1}_-(R)$ up to $8^{n-1}_-(r)$ when
$r<R$.  On the other hand:

\begin{thm} 
  \label{thm:higher-dim8s} 
  If $r \leq R$ and $a < b$, there does not exist an unknotted
  planar Lagrangian whose restriction to $\{ a \leq y_n \leq b \} \subset \rr^{2n}$
   gives a Lagrangian 
  cobordism $L$ 
  with 
   $$(L \cap \{ y_2 = a\} )= i_{a}(8^1_-(r) ),
  \qquad (L \cap \{ y_2 = b\} )= i_{b}(8^1_-(R) ).$$
\end{thm}

An extension of non-squeezing phenomena in
Theorem~\ref{thm:8squeeze} to other boundary slices
and to higher dimensions requires replacing the 
the area of one of the lobes  with a generating family  ``capacity'' of the 
slice.  The next subsection describes this new theory of  capacities.

% **********
\subsection{The Capacity Framework}
\label{ssec:cap-intro}

Proofs and extensions of Theorems~\ref{thm:8cobord},
\ref{thm:8squeeze}, and \ref{thm:higher-dim8s} involve a new notion of
a capacity.  The theory of capacities was preceded by Gromov's concept
of a symplectic radius \cite{gromov:hol} and was originally developed
through variational principles by Ekeland and Hofer
\cite{ek-h:symp-top-ham-dyn-1, ek-h:symp-top-ham-dyn-2}.  Viterbo
\cite{viterbo:generating} gave an alternative definition of a capacity
for Lagrangian submanifolds of cotangent bundles using the
finite-dimensional technique of generating families. In this paper, we
develop a new type of (generating family) capacity. Our techniques
build off of Viterbo's --- especially those he uses for families of
symplectic reductions in \cite[\S5]{viterbo:generating} --- but
produce capacities that allow us to study slices of an unknotted
planar Lagrangian.  More precisely, for each slice $L_a$ of a
unknotted planar Lagrangian $L \subset \rr^{2n}$ at a generic height
$a$, we define two lower and two upper capacities:
\begin{equation*} c^{L, a}_\pm : H^*(L_a) \to (-\infty, 0], \qquad
  C^{L,a}_\pm: H^*(L_a) \to [0, \infty).
\end{equation*} 
By analogy with the properties of capacities for subsets of
$\rr^{2n}$, we have the following theorem:

\begin{thm}
  \label{thm:properties} 
  The capacities satisfy the following properties:
  \begin{description}
  \item[Monotonicity] Suppose $a < b$, $0 \not \in [a,b]$, and $a$ and
    $b$ are generic heights of an unknotted planar Lagrangian $L$.
    Let $L_{[a,b]} = \bigcup_{t \in [a,b]} L_t$ be the cobordism between $L_a$
    and $L_b$ given by $L$, and let $j_t: L_t \to L_{[a,b]} $ be the inclusion
    map.  If $u \in H^*(L_{[a,b]} )$ then:
    \begin{align*} c^{L, a}_+(j^*_a u) &\leq c^{L, b}_+(j^*_bu), &
      C^{L, a}_+(j^*_au) &\leq C^{L, b}_+(j^*_bu), \\ c^{L,
        a}_-(j^*_au) &\geq c^{L, b}_-(j^*_bu), & C^{L, a}_-(j^*_au)
      &\geq C^{L, b}_-(j^*_bu).
    \end{align*}
    In any of the above relations, equality is possible only when both
    capacities equal $0$.
  \item[Continuity] Given $u \in H^*(L_{[a,b]})$, the function $c(t) =
    c^{L,t}(j^*_t u)$, where $c$ is any one of the four capacities, is
    continuous on $[a,b]$ (with removable discontinuities at
    non-generic levels) and piecewise differentiable.
  \item[Invariance] If $L^0$ and $L^1$ are unknotted planar
    Lagrangians that are isotopic via a compactly supported symplectic
    isotopy that fixes the slice at a generic height $a$, then for any
    cohomology class $u \in H^*(L^0_a) = H^*(L^1_a)$, we have
    $c_\pm^{L^0, a} (u) = c_\pm^{L^1, a}(u)$ and $C_\pm^{L^0, a} (u) =
    C_\pm^{L^1, a}(u)$.
  \item[Non-Vanishing] For any generic, nonempty slice $L_a$ of an
    unknotted planar Lagrangian $L$ and for any nonzero $u \in
    H^*(L_a)$, at least one of the four capacities $c_{\pm}^{L, a} (u)
    , C_{\pm}^{L, a} (u)$ is nonzero.
  \item[Conformality] If $L$ is an unknotted planar Lagrangian and
    $\beta L$ denotes the image of $L$ under the dilation $(\x, \y)
    \mapsto (\beta \x, \beta \y)$ then, for any generic height $a$,
    any of the four capacities $c$, and any $u \in H^*(L_a) \simeq H^*(\beta
    L_{\beta a})$,
    $$c^{\beta L,\beta a} (u) = \beta^2 c^{L,a}(u).$$
  \end{description}
\end{thm}

Although these capacities depend on the entire Lagrangian, it is
sometimes possible to compute these numbers only knowing the slice
$L_a$.  For example, suppose that the slice $L_a$ agrees with one of
our $\infty$-shaped curves $8_\pm^1(r)$ where the ``lobes'' each bound
a region with area of absolute value $A$.  Then for $0 \neq u \in
H^0(L_a)$ and $0 \neq v \in H^1(L_a)$, the following calculation shows
that these capacities carry geometric information about the slice:
\begin{center}
  \begin{tabular}{c|cccc} & $c^{L,a}_+(u)$ & $c^{L,a}_-(u)$ & $C^{L,a}_+(v)$ &
    $C^{L,a}_-(v)$ \\ \hline
    \includegraphics[height=15pt]{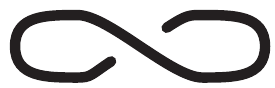} & $-A$ & 0 & 0 & $A$ \\
    \includegraphics[height=15pt]{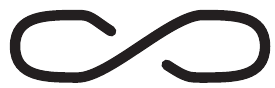} & 0 & $-A$ & $A$ & 0
  \end{tabular}
\end{center}
These calculations follow from Proposition~\ref{prop:4-dim-compute}.

A generalization of Theorem~ \ref{thm:8squeeze} to other shapes of
boundary curves and to higher dimensions can be formulated in terms
of capacities:
  
 \begin{thm} \label{thm:ht-cap}  Let $L \subset \rr^{2n}$ denote a
Lagrangian ball with  $\partial L = (L \cap  \{ y_n = a \})$ that has 
been obtained as the restriction of an unknotted planar Lagrangian
to $\{ y_n \geq a\}$ or to $\{ y_n \leq a \}$.     If the
  entire Lagrangian ball $L$ lies in the symplectic cylinder $\rr^{2n-2}
  \times I_{x_n} \times I_{y_n},$ for some intervals $I_{x_n} \subset \rr$
  of length $\ell$ and $I_{y_n} \subset \rr$ of length $w$, then for any $u \in H^*(\partial L)$:
  $$\ell \cdot w \geq |c^{L,a}(j^*_a u)|,$$
  where $c$ is any of the four capacities.
  \end{thm} 

\noindent
After rearranging the inequality as $w \geq \frac{1}{\ell}
|c^{L,a}(j^*_a u)|$, this theorem can be viewed as a ``directionally
sensitive'' version for planar Lagrangians of the following result by
Viterbo \cite{viterbo:cap-in-cotangent}: a Lagrangian that is
Hamiltonian isotopic to the zero section of $T^*T^n$ and is contained
in the unit disk bundle for some metric $g$ has a bounded generating
family length $\gamma(L)$.

Beyond these qualitative and quantitative results, 
the theory of capacities has a number of other extensions and
potential applications.  For example, for a closed manifold $B$, it is
possible to define capacities for slices of any Lagrangian submanifold
of $T^*(B \times \rr)$ that has a generating family; the statements
and proofs go through almost without modification.  Further, as 
discussed in Section~\ref{ssec:ft}, the capacities can be thought of as
fitting into Eliashberg, Givental, and Hofer's Symplectic Field Theory
(SFT) framework, the latest generation of pseudo-holomorphic curve
invariants \cite{egh}.  Like a TQFT structure, an SFT-type invariant
would assign a group to each slice and a homomorphism to each
Lagrangian cobordism between the slices.  The generating family
capacities defined in this paper follow a similar structure: a
capacity assigns a real number to each cohomology class of a slice,
and Monotonicity implies that each Lagrangian cobordism gives rise to
a relation between capacities.  In fact, in the process of defining
the capacities, we will have associated filtered groups to each slice.
The maps used to prove Monotonicity will turn into filtered
homomorphisms between these groups.  This framework has possible
applications to detecting knotting phenomena in Lagrangian surfaces
with fixed boundary, a question parallel to recent progress in the
study of Lagrangian cobordisms between Legendrian knots
\cite{chantraine, ehk:leg-knot-lagr-cob}.

The remainder of the paper is organized as follows: after reviewing
the basics of the theory of generating families in Section
\ref{sec:background}, we construct and investigate the geometry of
generating families for slices of unknotted planar Lagrangians in
Section \ref{sec:gf-lc}.  We define the capacities for a slice in
Section \ref{sec:capacity} using the Morse theory of a ``difference
function'' associated to a generating family of a slice.  In Section
\ref{sec:properties}, we prove the properties of the capacities listed
in Theorem~\ref{thm:properties}.  In Section~\ref{sec:apps}, we apply the capacity framework to
prove
Theorems~\ref{thm:8cobord} -- \ref{thm:higher-dim8s} and Theorem~\ref{thm:ht-cap}
and  discuss how to
fit the objects used to define the capacities into the framework of a
field theory.

\subsection*{Acknowledgments}

We thank Yasha Eliashberg for the inspiration for the original project
and for fruitful conversations once the project was underway.

% ********************
\section{Background on Generating Families}
\label{sec:background}

As described above, we will develop capacities for slices of generic,
unknotted planar Lagrangians in $T^*(\rr^n)$.  The foundation of our
construction is that all such Lagrangians have essentially unique
generating families associated to them.

Given a function $f: M \to \rr$, the graph of its differential is a
Lagrangian submanifold of $T^*M$.  The idea of generating families is
to extend this construction by looking at real-valued functions on
$M \times \rr^N$, for some potentially large $N$.  Suppose that we
have a smooth function $F: M \times \rr^N \to \rr$ such that $0$ is a
regular value of the map $(\frac{\partial{F}}{\partial e_1}, \dots,
\frac{\partial{F}}{\partial e_N}) : M \times \rr^N \to \rr^N$.  We
define $\Sigma_F$, the \textbf{fiber critical locus of F}, to be the
$m$-submanifold
$$ \Sigma_F := \left\{ (x, \e) \in M \times \rr^N  : 
  \pd{F}{e_i} (x, \e) = 0, \text{for } i = 1,2, \dots, N\right\}.$$ Define an
immersion $i_F : \Sigma_F \to T^*M$ in local coordinates by
\begin{equation} \label{eqn: crit_inclusion} i_F(x, \mathbf{e}) =
  \left(x, \pd{F}{x}(x, \mathbf{e} ) \right),
\end{equation} 
and let $L$ be the image of $i_F$.  It is easy to check that $L$ is
indeed Lagrangian.  We say that $F$ \textbf{generates} $L$, or $F$ is
a \textbf{generating family (of functions)} for $L$.  Note that we can
write $L$ as:
$$L = \left\{ \left(x, \pd{F}{x}(x, \e) \right) : 
  \pd{F}{\e} (x, \e) = 0 \right\}.$$ 
 Although the Lagrangian generated by $F$ may in
general be immersed, we will start with an embedded
Lagrangian and show that there exists a generating family for it.
 
\begin{exam} \label{exam:cylinder_construct} Via the theory of
  generating families, we can explicitly construct a flat-at-infinity
  (and thus unknotted) planar Lagrangian in $\rr^4$ whose restriction
  gives the cobordisms between the curves in Figure
  \ref{fig:basic-slices}(a), (c).  To construct such a Lagrangian for
  the curves in (a), choose $0 < \varepsilon < \bot < \top < K$, and
  consider $F: \rr^2 \to \rr$ given by
  $$F(x_1, x_2) = \ell(x_2) q(x_1) - d(x_1) c(x_2),$$
  where $\ell, q, d$ and $c$ are compactly supported functions such that,
  on the square $|x_1|, |x_2| \leq \sqrt{K-\varepsilon}$,
  $$ \ell(x_2) = x_2, \quad q(x_1) = K - x_1^2, \quad d(x_1) = 1, 
  \quad c(x_2) = \frac13 x_2^3,$$
  and, on the complement of this square,
  $$ \ell'(x_2) < 1, \quad q(x_1) < \varepsilon, \quad d(x_1) < 1, \quad
  -c'(x_2) < \bot-\varepsilon .$$ Then $F$ generates the
  flat-at-infinity planar Lagrangian $L$ satisfying
  \begin{align*} L = \{ (x_1, y_1, x_2, y_2) : \quad &y_1 =
    \ell(x_2)q'(x_1) - d'(x_1) c(x_2), \\ &y_2 = \ell'(x_2) q(x_1) -
    d(x_1) c'(x_2) \ \},
  \end{align*} and it is easy to verify that we get the desired  
  $8^1_-$ curves as pictured in Figure
  \ref{fig:basic-slices}(a): for $\bot
  \leq a \leq \top$,
  \begin{align*}
   L \cap \{ y_2 = a \} &= 
  \{ (x_1, x_2, -2x_1x_2, a ) : x_1^2 + x_2^2 = K - a \}, \\
  &= i_a( 8^1_-(\sqrt{K-a})).
  \end{align*}

  To construct an unknotted planar Lagrangian between the
  $8^1_+$  curves in Figure \ref{fig:basic-slices}(c), simply consider
  the function $G = -F$, and the desired curves occur for $-\tau \leq a
  \leq -\beta$.  
  
  These constructions can be easily generalized to higher dimensions:
  we construct a Lagrangian of the form $\Gamma_{dF}$ where
  $F: \rr^{n} \to \rr$ is compactly supported and is
  of the form
    $$F(x_1, \dots, x_n) = \ell(x_n) q(x_1, \dots, x_{n-1}) - d(x_1, \dots, x_{n-1}) c(x_n).$$
  Now, on an $n$-dimensional cube, 
  \begin{align*}
   &\ell(x_n) = x_n, \quad q(x_1, \dots, x_{n-1}) = 
  K - x_1^2-\dots - x_{n-1}^2,  \\
  &d(x_1,\dots, x_{n-1}) = 1, 
  \quad c(x_n) = \frac13 x_n^3,
   \end{align*}
   and outside this cube, the functions $\ell, q, d, c$ satisfy analogous conditions
   to those given in the $n = 2$ case above.  Then $\Gamma_{dF}$ is flat-at-infinity and, for 
   $a > \beta$,  the $L_a$ slices  agree with $8^{n-1}_-(\sqrt{K-a})$.    Also note that
   $\Gamma_{dF}$ is unknotted: it is not hard to explicitly construct a
   symplectic isotopy taking  $\Gamma_{dF}$ to the zero section.     
  $\diamond$
\end{exam}

If there is a generating family for a given Lagrangian $L$ then it is
easy to see that it is not unique: if $F: M \times \rr^N \to \rr$ generates
$L$ then so does, for example, $F': M \times \rr^{N+1} \to \rr$, where
$F'(x, e_1, \dots, e_N, e_{N+1}) = F(x, e_1, \dots, e_N) +
e_{N+1}^2.$ This is the first of three basic operations on a
generating family that will not change the Lagrangian that is
produced.

\begin{defn} Two generating families $F_i: M \times \rr^{N_i} \to
  \rr$, $i= 1, 2$, are \textbf{equivalent} if they can be made equal
  after the operations of addition of a constant, fiber-preserving
  diffeomorphism, and stabilization, where these operations are
  defined as follows:
  \begin{enumerate}
  \item Given a generating family $F: M \times \rr^N \to \rr$, let $Q:
    \rr^K \to \rr$ be a non-degenerate quadratic function.  Define $F
    \oplus Q: M \times \rr^N \times \rr^K \to \rr$ by $F \oplus Q(x,
    \e, \e') = F(x, \e) + Q(\e')$.  Then $F \oplus Q$ is a
    \textbf{stabilization} of $F$.
  \item Given a generating family $F: M \times \rr^N \to \rr$ and a
    constant $C \in \rr$, $F' = F + C$ is said to be obtained from $F$ by
    \textbf{ addition of a constant}.
  \item Given a generating family $F: M \times \rr^N \to \rr$, suppose
    $\Phi: M \times \rr^N \to M \times \rr^N$ is a fiber-preserving
    diffeomorphism, i.e., $\Phi(x,\e) = (x, \phi_x(\e))$
    for diffeomorphisms $\phi_x$ . Then $F' = F \circ \Phi$ is said to be
    obtained from $F$ by a \textbf{fiber-preserving diffeomorphism}.
  \end{enumerate}
\end{defn}

By construction, these generating families are defined on non-compact
domains. Analytically, it is convenient to consider functions that are
well-behaved outside of a compact set. A common convention has been to
consider generating families $F: M \times \rr^N \to \rr$ that are
\textbf{ quadratic-at-infinity}. This means that outside of a compact
set in $M \times \rr^N$, $F(x,\e) = Q(\e)$, where $Q$ is a
non-degenerate quadratic function. See, for example, Viterbo
\cite{viterbo:generating} and Th\'eret \cite{theret:viterbo}. Another
useful concept is to consider generating families $F : M \times \rr^N
\times \rr^K \to \rr$ that are \textbf{linear-quadratic-at-infinity}:
this means that outside a compact set in $M \times \rr^N \times
\rr^K$, $F(x, \mathbf{l}, \e) = J(\mathbf{l}) + Q(\e)$, where $J$ is a
nonzero linear function of $\mathbf{l}$ and $Q$ is a non-degenerate
quadratic function of $\e$; see, for example, \cite{lisa-jill}.

The following theorem gives  valuable existence and uniqueness
results for the unknotted planar Lagrangians considered in this paper.
The existence portion was proved by Sikorav in \cite{sikorav:gen-fn}
via the ``broken geodesic method'' of Laudenbach and Sikorav
\cite{Laudenbach+Sikorav}; existence can also be proved by a formula
devised by Chekanov; see, for example, \cite{traynor:shomology}.  The
uniqueness portion was proved by Viterbo \cite{viterbo:generating}
with precise details given by Th\'eret \cite{theret:viterbo}.

\begin{thm} [Existence and Uniqueness of Generating
  Families] \label{thm:2n-exist+unique} For $t \in [0,1]$, let $L_t
  \subset T^*\rr^n$ be the image of the zero section, $L_0$, under a
  compactly supported Hamiltonian isotopy $\phi_t$ of $T^*\rr^n$.
  Then:
  \begin{enumerate}
  \item There exists a quadratic-at-infinity generating family for
    $L_t$: if $F$ is any quadratic-at-infinity generating family for
    $L_0$ then there exists a path of quadratic-at-infinity generating
    families $F_t: \rr^n \times \rr^N \to \rr$ for $L_t = \phi_t(L_0)$
    so that $F_0$ is a stabilization of $F$ and $F_t = F_0$ outside a
    compact set.
  \item Any two quadratic-at-infinity generating families for $L_t$
    are equivalent.
  \end{enumerate}
\end{thm}

\begin{rem} \label{rem:gf-deform} In fact, the existence portion of
  Theorem~\ref{thm:2n-exist+unique} can be formulated more generally
  as a Serre fibration \cite{theret:viterbo} and applies to $L_t =
  \phi_t (L)$, where $L$ is any Lagrangian that has a
  quadratic-at-infinity generating family.
\end{rem}

\begin{rem} The results quoted above are actually stated for
  Hamiltonian isotopies of the zero-section of $T^*M$ where $M$ is
  compact.  Since we are dealing with compactly supported isotopies,
  we may think of our setting as $T^*S^n$.
\end{rem}

In our terminology, Theorem~\ref{thm:2n-exist+unique} says that any
unknotted planar Lagrangian in $\rr^{2n}$ has a unique (up to
equivalence) quadratic-at-infinity generating family associated to it.
In $\rr^4$, Theorem~\ref{thm:2n-exist+unique} applies to any
flat-at-infinity planar Lagrangian since Eliashberg and Polterovich
proved that such a Lagrangian must be unknotted \cite{ep:local-knots}.

 % ********************
\section{Generating Families and Difference Functions for Lagrangian
  Slices}
\label{sec:gf-lc}

In this section, we discuss generating families $F_a$ for slices
$L_a$; the Morse theory of the difference functions associated to such
generating families will be used in the next section to construct the
capacities.  The inspiration for the construction and use of a
difference function --- as opposed to Viterbo's direct use of the
generating family in his capacities --- comes from the second author's
study of two-component Legendrian links \cite{lisa:links}.  There, the
idea was to associate generating families $F_1(x, \e_1)$ and $F_2(x,
\e_2)$ to each component of the link and then study topological
invariants of the ``difference'' $\Delta$ of these generating
families, where $\Delta: M \times \rr^{N_1} \times \rr^{N_2} \to \rr$
is given by $\Delta(x, \e_1,\e_2) = F_1(x, \e_1) - F_2(x, \e_2)$.  The
critical points of the difference function pick out intersections of
the Lagrangian projections of the components of the link.  Here, we
will use a similar tactic to study the geometry of the slices $L_a$
using differences between $F_a$ and itself.

% **********
\subsection{Generating Families for Lagrangian Slices}
\label{ssec:gf-for-slice}

Let $L \subset \rr^{2n}$ be an unknotted planar Lagrangian.
By Theorem \ref{thm:2n-exist+unique}, ${L}$ has a quadratic-at-infinity
generating family $F: \rr^{n} \times \rr^N \to \rr$ that is unique up
to stabilization, fiber diffeomorphism, and an overall constant. Write
the coordinates on the domain of $F$ as $(\x, x_n, \e)$, where $\x \in
\rr^{n-1}$.

To study a slice $L_a$ of the Lagrangian, we first consider a new
generating family $F_a: \rr^{n-1} \times \rr^{1 + N} \to \rr$ given by
\begin{equation*} F_a(\x, x_n, \e) = F(\x, x_n, \e) - a\,x_n.  \label{eqn:F_a}
\end{equation*} 
For $F_a$, we are considering $x_n$ and $\e$ as fiber variables, and
so $F_a$ generates a Lagrangian in $T^*(\rr^{n-1}) = \rr^{2n-2}$ which
will, in fact, be a projection of our slice.  Let $\proj : \rr^{2n}
\to \rr^{2n-2}$ denote the projection $\proj(\x,x_n, \y, y_n) =
(\x,\y)$.

\begin{prop} If ${L}$ is transverse to the hypersurface $\{y_n =
  a\}$ then the function $F_a$ is a linear-quadratic-at-infinity
  generating family for the exact Lagrangian immersion $\proj(L_a)$.
\end{prop}

\begin{proof} To see that $F_a$ is linear-quadratic-at-infinity, note
  that since $F$ is equal to a non-degenerate quadratic $Q(\e)$ outside
  a compact set in $\rr^{n-1} \times \rr^{1+N}$, we have $F_a = -ax_n +
  Q(\e)$ there as well.  The fact  that $0$ is a regular value
  of $\left( \pd{F}{e_1}, \dots, \pd{F}{e_N} \right) : \rr^{n}
  \times \rr^{N} \to \rr^{N}$, and the hypothesis that  
  $L$ is transverse to $\{
  y_n = a \}$ will guarantee that $0$ is a regular value of $\left(
    \pd{F_a}{x_n}, \pd{F_a}{e_1}, \dots, \pd{F_a}{e_N} \right) : \rr^{n-1}
  \times \rr^{1+N} \to \rr^{1+N}$, and hence that $F$ generates \emph{some}
  Lagrangian.
To see that $F_a$ generates the claimed Lagrangian, we compute that
  $F_a$ generates the following set:
  \begin{equation*}
    \left\{ \left( \x, \pd{F_a}{\x} (\x,x_n,
        \e) \right) : \pd{F_a}{x_n} (\x, x_n, \e) = 0, 
      \pd{F_a}{\e} (\x, x_n, \e ) = 0 \right\}.
  \end{equation*}
  Rewriting this set using the fact that $F$ generates $L$ leaves us
  with:
  \begin{equation*}
    \{ (\x,\y ) : (\x, x_n, \y, a) \in L \},
  \end{equation*}
  which is simply the projection $\pi(L_a)$.
\end{proof}

\begin{rem} Since $\proj(L_a)$ is an exact Lagrangian immersion into
  $\rr^{2n-2}$, it lifts to an immersed Legendrian submanifold $\Lambda_a$ in
  $J^1(\rr^{n-1})$ with its usual contact structure.  
\end{rem}

% *********
\subsection{The Difference Function}
\label{ssec:diff-fn}  Define the \textbf{difference function}
\begin{equation*} \Delta_a: \rr^{n-1} \times \rr^{1+N} \times  \rr^{1+N} \to \rr
\end{equation*} by
\begin{equation*} \Delta_a(\x, x_n, \e, \tilde{x}_n,
  \tilde{\e}) = F_a(\x, x_n, \e) - F_a(\x,
  \tilde{x}_n, \tilde{\e}).
\end{equation*} 
We will see that, for generic $F$, $\Delta_a$ is a Morse-Bott
function.  	A {\bf generic Lagrangian} will be one that has a generating family
whose difference function satisfies a
Morse-Bott condition.  In this and the next section, we will
always be assuming that we are working with a generic
Lagrangian.  The results stated in Theorems~\ref{thm:8cobord} -- \ref{thm:higher-dim8s}
and Theorem~\ref{thm:ht-cap} will hold for arbitrary Lagrangian cobordisms
since any Lagrangian  can be perturbed into
a generic Lagrangian.

The capacities of a slice $L_a$ will be constructed using
the Morse theory of $\Delta_a$, so it is important to identify its
critical points.

\begin{lem} \label{lem:crit-pts} The critical points of $\Delta_a$ are
  of two types:
  \begin{enumerate}
  \item For each double point $(\x, \y)$ of $\proj(L_a)$, there are
    two critical points $(\x, x_n, \mathbf{e}, \tilde x_n, \tilde
    {\mathbf{e}} ) $ and $(\x, \tilde x_n, \tilde{\mathbf{e}}, x_n,
    \mathbf{e})$ whose critical values are either both $0$ or are $\pm
    v$, for some $v \neq 0$.  
  \item The set
    \begin{equation*} C_a = \bigl\{ (\x, x_n, \mathbf{e}, x_n,
      \mathbf{e})\; : \; (\x, x_n, \mathbf{e}) \in \Sigma_{F_a} \bigr\}
    \end{equation*} is a critical submanifold of
    $\Delta_a$ with critical value $0$.
  \end{enumerate}
  For generic $F$, these critical points and submanifolds are
  non-degenerate and $C_a$ has index $1 + N$. 
\end{lem}

\begin{proof} It is straightforward to calculate that at a critical
  point $(\x, x_n, \e, \tilde{x}_n, \tilde{\e})$ of $\Delta_a$, the
  points $(\x, x_n, \e)$ and $(\x, \tilde{x}_n, \tilde{\e})$ both lie
  in $\Sigma_F$.  
%	  \footnote{In more detail, to find the critical points
%	    $(\x, x_n, \e, \tilde{x}_n, \tilde{\e})$ of $\Delta_a$, notice that
%	    \begin{align} \pd{\Delta_a}{\e}(\x, x_n, \e, \tilde{x}_n,
%	      \tilde{\e}) = 0 &\iff (\x, x_n, \e) \in \Sigma_F \\
%	      \pd{\Delta_a}{\tilde\e}(\x, x_n, \e, \tilde{x}_n, \tilde{\e}) = 0
%	      &\iff (\x, \tilde{x}_n, \tilde\e) \in \Sigma_F
%	    \end{align}} 
    Further, their images in $L$ both lie in $L_a$ and
  have the same $\y$ coordinate.
%	  \footnote{This is because:
%	    \begin{align} \pd{\Delta_a}{x_n} = 0 &\iff \pd{F}{x_n} =
%	      a, \label{eqn:d-delta-1}\\ \pd{\Delta_a}{\tilde{x}_n} = 0 &\iff
%	      \pd{F}{\tilde{x}_n} = a,
%	      \label{eqn:d-delta-2} \\ \pd{\Delta_a}{\x} = 0 &\iff
%	      \pd{F}{\x} (\x, x_n, \e) = \pd{F}{\x} (\x, \tilde{x}_n,
%	      \tilde{\e}).  \label{eqn:d-delta-3}
%	    \end{align} Equations (\ref{eqn:d-delta-1}) and
%	    (\ref{eqn:d-delta-2}) tell us that the aforementioned pair of points
%	    in $L$ has $y_n = a$ and $\tilde{y}_n = a$;
%	    Equation~(\ref{eqn:d-delta-3}) tells us that for this pair of points
%	    $\y = \tilde{y}_1$.}  
    Thus, the critical points of $\Delta_a$
  correspond to pairs of points in $L$ of the form $(\x, x_n, \y, a)$
  and $(\x, \tilde{x}_n, \y, a)$.  The existence of the two types of
  critical points asserted in the lemma follows: a set coming from
  double points of the projection $\proj(L_a)$, and a set $C_a$ coming
  from any point in $L_a$ paired with itself.

  To understand the non-degeneracy claim, note that if $(\x, x_n, \e,
  \tilde x_n, \tilde \e)$ is a critical point of $\Delta_a$ then $(\x,
  x_n, \e), (\x, \tilde x_n, \tilde \e) \in \Sigma_{F_a}$.  So, after
  a fiber-preserving diffeomorphism of $\rr^{n-1} \times \rr^{1+N}$,
  we may assume that in a neighborhood of a critical point $(\x, x_n,
  \e, \tilde x_n, \tilde \e)$ of $\Delta_a$, there exist functions
  $g(\x)$, $h(x_n,\e)$, $\tilde g(\x)$, and $\tilde h(\tilde x_n,
  \tilde \e)$ such that
  $$F_a(\x, x_n, \e) =g(\x) + h(x_n, \e), \quad 
  F_a(\x, \tilde x_n, \tilde \e) =\tilde g(\x) + \tilde h(\tilde x_n,
  \tilde \e).$$ If $A$ (resp.\ $\tilde{A}$) is the Hessian $d^2h$ at
  $(x_n, \e)$ (resp.\ $d^2\tilde h$ at $(\tilde{x}_n, \tilde{\e})$)
  then:
  \begin{equation}
    \label{eqn:hessian-delta} 
    d^2\Delta_a (\x, x_n, \e, \tilde x_n, \tilde \e)= \begin{bmatrix}  
      d^2g - d^2 \tilde{g}(\x) & \mathbf{0} & \mathbf{0} \\ \mathbf{0} 
      & A(x_n, \e) & \mathbf{0} \\ \mathbf{0} & \mathbf{0} & 
      -\tilde{A}(\tilde x_n, \tilde \e)
    \end{bmatrix}.
  \end{equation} Since $F_a$ is a generating family,
  $A$ and $\tilde{A}$ are non-degenerate.  At an isolated
  critical point, for generic $F$, the upper left entry is non-degenerate, so
  $d^2\Delta_a$ is non-degenerate.  For a point in $C_a$, the kernel of
  $d^2\Delta_a$ is spanned by vectors of the form $[\mathbf{v},
  \mathbf{0}, \mathbf{0}]$, which
  coincides with $TC_a$;
%	  \footnote{This follows from the fact that
%	    $\partial_{\e} F_a = \partial_{\e} h$, which is independent of $\x$,
%	    and hence $C_a$ lies entirely in the $\x$ direction.}  
  the non-degeneracy follows, as does the claim about the index of the
  critical
  submanifold.% since $-\tilde A(\tilde x_n, \tilde \e) = - A(x_n, \e)$.
\end{proof}

\begin{rem} \label{rem:l-diff}
  We can similarly define a difference function for the generating
  family $F$ for $L$ by:
  \begin{equation*}
    \delta(\x, x_n,\e, \tilde{\e}) = F(\x, x_n,\e) - F(\x, x_n, \tilde{\e}).
  \end{equation*}
  By the same proof as above, the assumption that $L$ is embedded
  implies that all critical points of  $\delta$ have critical value zero.
\end{rem}

\begin{exam} \label{exam:explicit_crit_values} For the generating
  family $F: \rr^2 \to \rr$ constructed in Example
  \ref{exam:cylinder_construct}, it is straightforward to check that
  for $ \bot < a < \top$, all critical points of $\Delta_a$ occur when
  $|x_1|, |x_2|, |\tilde x_2| < \sqrt {K - \varepsilon}$.  In this
  region, we have $F(x_1, x_2) \equiv x_2(K-x_1^2) - \frac13x_2^3$, so
  the isoloated critical points are:
  \begin{align*}
    q_- &= \left(0, \sqrt{K-a}, -\sqrt{K-a}\right), & \text{and} \\
    q_+ &= \left(0, -\sqrt{K-a}, \sqrt{K-a}\right),
  \end{align*}
  while the critical submanifold is given by $\bigl(s, t, t\bigr)$,
  where $s^2 + t^2 = K - a$.  Using the explicit definition of $F$, we
  find that the critical values are:
  $$\Delta_a(q_-) = \frac43(K-a)^{\frac32}> 0, \quad \Delta_a(q_+) = -\frac43(K-a)^{\frac32} < 0,$$
  while the indices are given by
  $$ \ind(q_-) = 3, \quad \ind(q_+) = 0.$$
  Furthermore, it can be explicitly seen that $C_a = \{ (s, t, t ) :
  s^2 + t^2 = K-a \}$ is a non-degenerate critical submanifold of
  critical value $0$ and index $1$.  $\diamond$
\end{exam}

% **********
\subsection{Morse Theory for $\Delta_a$ on Split Domains}

In order to detect information about the sign of the crossings
$\proj(L_a)$, the capacities will be defined using intersections of
the sublevel sets of $\Delta_a$ with ``positive'' and ``negative''
half-spaces of the domain. Since the boundary between these
half-spaces is not a level set of $\Delta_a$, we cannot directly use
the usual Morse-theoretic constructions, but we shall see that certain
techniques still work.

To set notation, denote the sublevel sets of $\Delta_a$ by:
\begin{equation*} \Delta_a^\lambda = \bigl\{ (\x, x_n, \mathbf{e},
    \tilde{x}_n, \tilde{\mathbf{e}}) \in \rr^{n-1} \times
  \rr^{1+N} \times \rr^{1+N} \; : \; \Delta_a(\x, x_n, \mathbf{e},
    \tilde{x}_n, \tilde{\mathbf{e}}) \leq \lambda
  \bigr\}.
\end{equation*} 
The distinction between positive and negative capacities comes from a
splitting of the domain $\rr^{n-1} \times \rr^{1+N} \times \rr^{1+N}$ into
positive and negative pieces as follows:
\begin{equation*}
  \begin{split} \mathcal{P}_+ &= \bigl\{ (\x, x_n, \mathbf{e},
    \tilde{x}_n, \tilde{\mathbf{e}})\; : \; x_n \leq  \tilde{x}_n
    \bigr\}, \\ \mathcal{P}_- &= \bigl\{ (\x, x_n, \mathbf{e},
    \tilde{x}_n, \tilde{\mathbf{e}})\; : \; x_n \geq  \tilde{x}_n
    \bigr\}.
  \end{split}
\end{equation*} 
We divide sublevel sets into positive and negative pieces as well:
\begin{equation} \Delta_{a,\pm}^\lambda = \Delta_a^\lambda \cap
  \mathcal{P}_\pm. \label{eqn:split-sublevels}
\end{equation} 
The motivation for this splitting comes from the fact that
\begin{equation}   \Delta_a(\mathbf{x}) -
  \Delta_b(\mathbf{x}) = (a-b)(\tilde{x}_n - x_n).
  \label{eqn:difference}
\end{equation} In particular, on $\mathcal{P}_+$, $a < b$ implies
$\Delta_b \geq \Delta_a$, and thus $\Delta_b^\lambda \subset
\Delta_a^\lambda$; while on $\mathcal{P}_-$, $a < b$ implies
$\Delta_a^\lambda \subset \Delta_b^\lambda$.  Thus, for $a < b$ and
for any $\lambda$, we have
\begin{equation} \Delta_{b,+}^\lambda \subset \Delta_{a, +}^\lambda
  \text{ and } \Delta_{a,-}^\lambda \subset \Delta_{b,
    -}^\lambda.  \label{eqn:difference-inclusions}
\end{equation} This type of containment will be instrumental in the
proof of Monotonicity for the capacities.

The intersections of sublevel sets with half-spaces obey a central
Morse-theoretic lemma: if there is no critical value of $\Delta_a$ in
$[\sigma, \tau]$ then each half of the sublevel set at $\sigma$ is a
deformation retract of the corresponding half of the sublevel set at
$\tau$. The proof of the following lemma relies crucially on the
assumption that $L$ is embedded.

\begin{lem} \label{lem:def-retr} If there are no critical points of
  $\Delta_a$ in $\mathcal{P}_+$ with critical value in the interval
  $[\sigma, \tau]$ then $\Delta^\sigma_{a, +}$ is a deformation
  retract of $\Delta^\tau_{a,+}$.  A similar statement holds with
  ``$+$'' replaced by ``$-$''.
\end{lem}

\begin{proof} Let $ \mathcal{P}_0 = \{ (\x, x_n, \mathbf{e},
  \tilde{x}_n, \tilde{\mathbf{e}}) : x_n = \tilde{x}_n\}.$ The idea of
  the proof is to modify the gradient of $\Delta_a$ to an integrable
  vector field $X$ such that:
  \begin{enumerate}
  \item The derivative $X(\Delta_a)$ is positive and uniformly bounded
    away from $0$ on $\Delta_a^{-1}([\sigma, \tau])$, and
  \item $X$ is tangent to $\mathcal{P}_0$ (and hence its flow
    preserves $\mathcal{P}_\pm$).
  \end{enumerate} Following the negative flow of such a vector field
  clearly gives the desired deformation retract.

  To construct the vector field $X$, first decompose $\nabla \Delta_a$
  into a component $X_T$ parallel to $\mathcal{P}_0$ and a component
  $X_N$ normal to $\mathcal{P}_0$.  To define $X$, let
  $\beta_\epsilon(t)$ be a smooth nonnegative function that is equal
  to $1$ outside of $(-\epsilon, \epsilon)$ for some positive
  $\epsilon$ to be chosen, does not exceed $1$, and vanishes at $0$.
  Then let:
  \begin{equation*} X = X_T + \beta_\epsilon(\tilde{x}_n - x_n) X_N.
  \end{equation*}

  It is clear by construction that $X$ is tangent to $\mathcal{P}_0$.
  To prove that $X(\Delta_a)$ is uniformly bounded away from $0$ on
  $\Delta_a^{-1}([\sigma, \tau])$, split the domain into an
  $\epsilon$-neighborhood of $\mathcal{P}_0$ and its complement. We
  will show that on each piece, there exists an $r>0$ such that
  $X(\Delta_a) \geq r$.

  First consider an $\epsilon$-neighborhood of $\mathcal{P}_0$, with
  $\epsilon$ still to be determined. It is easy to check that the
  restriction of $\Delta_a$ to $\mathcal{P}_0$ is the difference
  function $\delta$ for the original Lagrangian $L$ as in Remark~\ref{rem:l-diff}. 
   Thus, along
  $\mathcal{P}_0$, we have $X_T = \nabla \delta$.  The embeddedness of
  $L$ implies that $\delta$ has no critical values in $[\sigma,
  \tau]$.  Further, since $\delta$ is the difference of two functions
  that are quadratic-at-infinity, it is straightforward to prove that
  there is an $r>0$ such that $\nabla \delta(\Delta_a) = \|\nabla
  \delta \|^2 \geq 2r$ on $\delta^{-1}([\sigma, \tau])$.
%	  \footnote{To see this, suppose that
%	    $F(\x, x_n, e) = Q(e)$ outside a compact set $K_x \times K_e$
%	    containing the origin.  To prove $\|\nabla \delta \|^2 \geq 2r$,
%	    we split $\delta^{-1}([\sigma, \tau])$ into three pieces: its
%	    intersection with $K_x \times K_e \times K_e$, its intersection
%	    with $(\rr^2 \setminus K_x) \times K_e \times K_e$, and everything
%	    else.  In the first case, since $\delta$ has no critical points
%	    and the set in question is compact, $\|\nabla \delta \|^2$ is
%	    uniformly bounded below.  In the second case, we have $\delta(\x,
%	    x_n, e, \tilde{e}) = Q(e) - Q(\tilde{e})$, so
%	    \begin{align} \| \nabla \delta \|^2 &= \| \nabla Q(e) \|^2 + \|
%	      \nabla Q(\tilde{e}) \| ^2 \\ &\geq C (\|e\|^2 +
%	      \|\tilde{e}\|^2) \label{eqn:Q-bound}
%	    \end{align} for some $C >0$ since $Q$ is non-degenerate.  Since
%	    $\delta(\x,x_n,e,\tilde{e}) \in [\sigma, \tau]$, at least one of $e$
%	    and $\tilde{e}$ must lie a bounded distance away from $0$ in the set
%	    under consideration, and hence $\| \nabla \delta \|^2$ is also bounded
%	    away from $0$.  In the final case, at least one of the terms in
%	    $\delta$ is equal to $Q(e)$ or $Q(\tilde{e})$ for $e$ or $\tilde{e}$
%	    bounded away from $0$, so by (\ref{eqn:Q-bound}), $\| \nabla \delta
%	    \|^2$ is bounded away from $0$ on this set as well.} 
    Choose
  $\epsilon>0$ so that $\| X_T \|^2 \geq r$ on an
  $\epsilon$-neighborhood of $\mathcal{P}_0$; this immediately implies
  that $\| X \|^2 \geq r$ there as well.
Outside that neighborhood, we have $X(\Delta_a) = \|\nabla
  \Delta_a\|^2$.  As with $\delta$, since there are no critical values
  of $\Delta_a$ in $[\sigma, \tau]$ and $\Delta_a$ is the difference of
  two functions that are linear-quadratic at infinity, it is
  straightforward to prove that there is an $r>0$ such that $\|\nabla
  \Delta_a \|^2 \geq 2r$ on $\Delta_a^{-1}([\sigma, \tau])$.
%	  \footnote{In
%	    fact, the proof is easier than for $\delta$.  On the compact set
%	    associated with the linear-quadratic property of $\Delta_a$, $\|\nabla
%	    \Delta_a \|^2$ is bounded away from $0$ by hypothesis.  Outside that
%	    compact set, at least one of $a x_n$ or $a \tilde{x}_n$ appears as the
%	    only instance of $x_n$ or $\tilde{x}_n$ in the expression of
%	    $\Delta_a$, and hence $\|\nabla \Delta_a \|^2 \geq a^2$.  }
\end{proof}

% ********************
\section{Capacities for Lagrangian Slices}
\label{sec:capacity}

The goal of this section is to associate to a generic slice $L_a$  
and each element $u \in H^k(L_a)$
 four real numbers $c_\pm$,
$C_\pm$.
These capacities will be defined using
Morse-theoretic techniques applied to the split sublevel sets
$\Delta_{a,\pm}^\lambda$ and will satisfy the five properties listed in
Theorem~\ref{thm:properties}.

% **********
\subsection{Definition of the Capacities}
\label{sec:defn-cap}

To define the capacities, we will need to examine maps between the
relative cohomology groups of the sublevel sets $\Delta^\lambda_{a,
  \pm}$ and the cohomology groups of $L_a$.  As has been the
convention, assume that the domain of $\Delta_a$ is $\rr^{n-1} \times
\rr^{1+N} \times \rr^{1+N}$.  At some point, we will need to assume
that $ N \geq n-1$, a condition that can be guaranteed by
stabilization.  Suppose throughout this section that $\lambda < -\eta
< 0 < \eta < \Lambda$ and that $\eta$ has been chosen small enough so
that $0$ is the only critical value of the difference function
$\Delta_a$ in $[-\eta, \eta]$.  By Lemma~\ref{lem:def-retr}, the
precise choice of $\eta$ is immaterial.  For all $k \in \zz_{\geq 0}$, we will
define maps $i_\pm^*$, $p_a^\lambda$, $D_a^\Lambda$, and $\mathcal I$
whose domains and ranges are related as follows:
\begin{equation*} \label{eqn:cap-maps}
  \xymatrix@C=0pt{
    & H^k(L_a) \ar[d]^{\mathcal{I}} & \\
    & H^{k+N+1}(\Delta_a^{\eta},
    \Delta_a^{-\eta}) \ar[dl]_{p_a^\lambda} \ar[dr]^{D_a^\Lambda} & \\
    H^{k+N+1}(\Delta_a^\eta, \Delta_a^\lambda) \ar[d]^-{i^*_\pm} & & 
    H^{k+N+2}(\Delta_a^{\Lambda}, \Delta_a^{\eta})
    \ar[d]^-{i_{\pm}^*} \\
    H^{k+N+1}(\Delta_{a, \pm}^\eta, \Delta_{a, \pm}^\lambda) & &
    H^{k+N+2}(\Delta_{a,\pm}^{\Lambda},
    \Delta_{a,\pm}^{\eta}).
  }
\end{equation*}

The maps $i_\pm^*$ come from a Mayer-Vietoris sequence.  One reason
that the splitting using $\mathcal{P}_\pm$ is nice is that
the relative cohomology of two sublevel sets is completely determined
by the relative cohomologies of their splittings.

\begin{lem} \label{lem:m-v-splitting} Suppose that $0 \notin
  [\sigma, \tau]$.  Then, for any $k \in \zz_{\geq 0}$, the natural inclusions
  induce an isomorphism:
  \begin{equation*} (i_+^*, i_-^*): H^k(\Delta_{a}^\tau,
    \Delta_{a}^\sigma) \to H^k(\Delta_{a,+}^\tau, \Delta_{a,+}^\sigma)
    \oplus H^k(\Delta_{a,-}^\tau, \Delta_{a,-}^\sigma).
  \end{equation*} 
\end{lem}

\begin{proof} This will follow from a Mayer-Vietoris argument.  As in
  the proof of Lemma~\ref{lem:def-retr}, let $\mathcal{P}_0 = \{ (\x,
  x_n, \mathbf{e}, \tilde{x}_n, \tilde{\mathbf{e}}) : x_n =
  \tilde{x}_n\}$ and let $\Delta_{a,0}^\lambda = \Delta_{a,+}^\lambda
  \cap \Delta_{a,-}^\lambda$; that is, $\Delta_{a,0}^\lambda =
  \Delta_a^\lambda \cap \mathcal P_0$.
  % \footnote{Since $\mathcal{P}_0$ is always transverse to
  %   $\operatorname{Int} \Delta^\lambda_a$ and to its boundary,
  %   $\partial \Delta^\lambda_a$, for generic $a$, $\mathcal P_0$ has
  %   a proper neighborhood in the manifold-with-boundary
  %   $\Delta^\lambda_a$.  Thus, $\mathcal P_0$ is a neighborhood
  %   deformation retract, and we can substitute the $\pm,0$ parts for
  %   neighborhoods thereof.}
  It suffices to prove that $H^k(\Delta_{a,0}^\tau,
  \Delta_{a,0}^\sigma) = 0$.  In fact, we will show that
  $\Delta_{a,0}^\sigma$ is a deformation retract of
  $\Delta_{a,0}^\tau$.  Since $\Delta_a|_{\mathcal{P}_0}$ coincides
  with the difference function $\delta$ for $F$ introduced in
  Remark~\ref{rem:l-diff}, and since the embeddedness of $L$ implies since $\delta$ has no nonzero critical values, we can use $-\nabla \delta$ to
  deformation retract $\Delta_{a,0}^\tau$ to
  $\Delta_{a,0}^\sigma$.  
\end{proof}

The maps $p_a^\lambda$ and $D_a^\Lambda$ are defined from examining
long exact sequences of triples:

\begin{defn} \label{defn:connecting maps} For each $k \in \zz_{\geq
    0}$, in the long exact sequence of the triple $(\Delta_a^\eta,
  \Delta_a^{-\eta}, \Delta_a^\lambda)$, let $p^\lambda_a$ be the
  projection homomorphism:
  \begin{equation*} p^\lambda_a:  H^{k+N+1}(\Delta_a^{\eta},
    \Delta_a^{-\eta}) \to H^{k+N+1}(\Delta_a^{\eta}, \Delta_a^\lambda).
  \end{equation*}   Similarly, in the
  long exact sequence of the triples $(\Delta_a^\Lambda,
  \Delta_a^{\eta}, \Delta_a^{-\eta})$, let
  \begin{equation*} D^\Lambda_a: H^{k+N+1}(\Delta_a^{\eta},
    \Delta_a^{-\eta}) \to H^{k+N+2}(\Delta_a^{\Lambda}, \Delta_a^{\eta})
  \end{equation*} be the connecting homomorphism.
\end{defn}

Lastly, we use a Gysin sequence to define the map $\mathcal I$.

\begin{lem} For $k \in \zz_{\geq 0}$, by stabilizing if necessary,
  assume that $k+ N \geq n-1$.  Then there exists an injective map
  $$\mathcal I: H^k(L_a) \to H^{k+N+1} (\Delta^\eta_a, \Delta^{-\eta}_a).$$
  Moreover, $\mathcal I$ is an isomorphism when there are no
  non-degenerate critical points with critical value $0$ and index $k
  + N + 1$.
\end{lem}

\begin{proof} Recall that $\Delta_a$ always has a non-degenerate
  critical submanifold $C_a$ diffeomorphic to $L_a$ with critical
  value $0$ and index $N+1$; suppose that there are $m$ non-degenerate
  critical points, $z_1, \dots, z_m$, of critical value $0$ and index
  $k + N+1$.  Let $W^-(C_a)$ denote the descending disk bundle of the
  the critical set $C_a$.  Standard Morse-Bott theory says that the
  homotopy type of $\Delta_a^\eta$ is obtained from that of
  $\Delta_a^{-\eta}$ by attaching one $(k+N+1)$-cell for each $z_i$
  and attaching the (oriented) disk bundle $W^-(C_a)$ along its
  bounding sphere bundle.  Thus we have an isomorphism:
  $$ H^{k+N+1}(\Delta_a^\eta, \Delta_a^{-\eta}) \simeq
  H^{k+N+1}(W^-(C_a), \partial W^-(C_a)) \oplus\zz^m.$$ Let $\rho:
  H^{k+N+1}(\Delta_a^\eta, \Delta_a^{-\eta}) \to
  H^{k+N+1}(W^-(C_a), \partial W^-(C_a))$ denote the associated
  (surjective) projection map with $\rho^{-1}$ the associated
  inclusion.  Lastly, to identify $H^{k+N+1}(W^-(C_a), \partial
  W^-(C_a))$ with $H^k(L_a)$, we consider the Gysin sequence
%	  \begin{multline} \dots \to H^{k+N+1}(C_a) \to H^{k+N+1}
%	    (W^-(C_a), \partial W^-(C_a)) \stackrel{\pi_*}{\longrightarrow} 
%	    H^k(C_a) \to \\H^{k+N+2}(C_a)\to \cdots. \hskip 2.2in
%	  \end{multline}
  \begin{align*} \dots \to &H^{k+N+1}(C_a) \to H^{k+N+1}
    (W^-(C_a), \partial W^-(C_a)) \stackrel{\pi_*}{\longrightarrow} 
    H^k(C_a) \to \\ &H^{k+N+2}(C_a)\to \cdots. \hskip 2.2in
  \end{align*}
  Since $C_a$ is diffeomorphic to the $(n-1)$-dimensional $L_a$, our
  hypothesis that $N \geq n-1$ implies that  $k+ N
  + 1 > n-1$, and thus $\pi_*$ is an isomorphism.  The desired
   injective map $\mathcal I$ is defined to be
  $(\rho)^{-1}\circ (\pi_*)^{-1}$.
\end{proof}

We now have the maps necessary for the definition of the
capacities: $$\varphi_{a, \pm}^\lambda : H^k(L_a) \to
H^{k+N+1}(\Delta_{a,\pm}^\eta, \Delta_{a,\pm}^\lambda)$$ is given by
$\varphi_{a, \pm}^\lambda = i^*_\pm \circ p^\lambda_a \circ
\mathcal{I}$, while $$\Phi_{a, \pm}^\Lambda : H^k(L_a) \to H^{k+N+2}
(\Delta_{a, \pm}^\Lambda, \Delta_{a, \pm}^\eta)$$ is given by
$\Phi_{a, \pm}^\Lambda = i_\pm^* \circ D_a^\Lambda \circ \mathcal
{I}$.

\begin{defn} \label{defn:capacity} For $u \in H^k(L_a)$, the {\bf
    positive and negative lower capacities}, $c_\pm^{L, a}(u) \in
  (-\infty, 0]$, are defined to be
  $$c_\pm^{L, a} (u) = \sup \{ \lambda<0 : \varphi_{a, \pm}^\lambda(u) = 0\};$$
  if the set on the right hand side is empty then $c_\pm^{L,a}(u) =
  0$.  For $u \in H^k(L_a)$, the {\bf positive and negative upper
    capacities}, $C_\pm^{L, a}(u) \in [0, \infty)$, are defined to be
  $$C_\pm^{L, a}(u) = \inf \{ \Lambda>0 : \Phi_{a, \pm}^\Lambda(u) \neq 0 \};$$
  if the set on the right hand side is empty then $C_\pm^{L,a}(u) =
  0$.
\end{defn}

% **********
\subsection{Foundational Properties}
\label{ssec:basic-prop}

In this section, we prove two important technical properties, the
first of which is that the capacities always occur at the critical
values of the difference function $\Delta_a$.

\begin{lem} \label{lem:cap-at-crit} For all $u \in H^k(L_a)$,
  $c_\pm^{L, a}(u)$ and $C_\pm^{L,a}(u)$ are critical values of
  $\Delta_a$.
\end{lem}

\begin{proof} We will prove the proposition for the capacities
  $c_\pm^{L,a}(u)$, as the proof for the capacities $C_\pm^{L,a}(u)$
  is entirely similar.

  First notice that if $\varphi_{a, \pm}^\lambda(u) \neq 0$, for all
  $\lambda$, then by definition $c_\pm^{L,a}(u) = 0$, and by Lemma
  \ref{lem:crit-pts}, $0$ is always a critical value of $\Delta_a$.
  Next, we will show that if $ \varphi_{a,\pm}^\lambda(u) = 0$ for
  some $\lambda$ that is a non-critical value of $\Delta_a$, then
  $\lambda < c_{\pm}^{L,a}(u)$.  Suppose there are no critical values
  of $\Delta_a$ in $[\lambda, \nu]$.  By Lemma~\ref{lem:def-retr}, the
  inclusion map of pairs
  \begin{equation*} i: (\Delta^{\eta}_{a, \pm},
    \Delta^\lambda_{a,\pm}) \to (\Delta^{\eta}_{a, \pm},
    \Delta^\nu_{a,\pm})
  \end{equation*} induces an isomorphism on relative cohomology.  The
  commutativity of the  diagram 
%	    \footnote{If $u = \phi^\lambda_{a,\pm}(x)$ for some $x \in H^n
%	    (\Delta^{-\eta}_{a, \pm}, \Delta^\lambda_{a,\pm})$, then the
%	    commutativity of the following diagram shows that $u =
%	    \phi^\nu_{a,\pm}( (i^*)^{-1}(x))$.}
  \begin{equation*} 
    \xymatrix@R=8pt{ & H^{k+N+1}(\Delta^{\eta}_{a,\pm},
      \Delta^{\lambda}_{a,\pm})  \ar[dd]^{(i^*)^{-1}} \\
      H^k(L_a) \ar[ur]^-{\varphi^\lambda_{a, \pm}}
      \ar[dr]_-{\varphi^\nu_{a, \pm}} & \\
      & H^{k+N+1}(\Delta^{\eta}_{a,\pm}, \Delta^{\nu}_{a,\pm})
      }
  \end{equation*}
  shows that if 
  $ \varphi^\lambda_{a, \pm}(u) = 0$ 
   then 
  $ \varphi^\nu_{a, \pm} (u) = 0$.
  Thus, we obtain $\lambda < c_\pm^{L, a}(u)$.
\end{proof}

\begin{cor} \label{cor:vanishing_capacities} Suppose $0 \neq u \in
  H^k(L_a)$.  Then $c_\pm^{L,a}(u) = 0 $ if and only if
  $\varphi_{a, \pm}^\lambda (u) \neq 0$
  % $u \notin \ker  \varphi_{a, \pm}^\lambda$ 
  for all $\lambda$, and $C_\pm^{L,a}(u) = 0 $ if and only if
  $\Phi_{a, \pm}^\Lambda(u) = 0$ for all $\Lambda$.
\end{cor}

\begin{proof} We will prove this for $c_\pm^{L,a}(u)$; the argument
  for $C_\pm^{L,a}(u)$ is analogous.

  By definition, if $\varphi_{a, \pm}^\lambda (u) \neq 0$,
  %$u \notin \ker \varphi_{a, \pm}^\lambda$
  for all $\lambda$, then $c_\pm^{L,a}(u) = 0$.  To show the converse, first
  let $V^-$ denote the set of negative critical values of $\Delta_a$,
  and let $m = \sup V^-$.  Note that $m < 0$ since the set of critical
  values of $\Delta_a$ is discrete (as $\Delta_a$ is
  linear-quadratic-at-infinity).  We will show that if $0 \neq u \in
  \ker \varphi_{a,\pm}^\lambda$ for some $\lambda$, then
  $c_{\pm}^{L,a}(u) \leq m$.  If not, then we can assume $u \in \ker
  \varphi_{a,\pm}^{\lambda}$ for some $m < \lambda < -\eta$.  But then
  by examining the long exact sequence of the triple $(\Delta_a^\eta,
  \Delta_a^{-\eta}, \Delta_a^\lambda)$, we find that the map
  $p_a^{\lambda}$ is injective, and thus $0 \neq u \in H^k(L_a)$
  cannot be in the kernel of $\varphi_{a, \pm}^{\lambda}$.
\end{proof}

The second important technical property is that the capacities are
independent of the choice of generating family $F$ of $L$.

\begin{lem} \label{lem:cap-well-def} For all $u \in H^k(L_a)$, the
  capacities $c_\pm^{L, a}(u)$ and $C_\pm^{L,a}(u)$ are independent of
  the generating family $F$ used to define $L$.
\end{lem}

\begin{proof} It suffices to show that the capacities are unchanged
  if $F$ is altered by the addition of a constant, a fiber-preserving
  diffeomorphism, or a stabilization.

  If $F'$ is obtained from $F$ by the addition of a constant  then the
  corresponding difference function $\Delta_a'$ agrees with the
  difference function $\Delta_a$, and thus the capacities are
  unchanged.
  
  Next, suppose that $F'$ is obtained from $F$ by a fiber-preserving
  diffeomorphism.  Then $F_a'(\x, x_n, \e) = F(\x, x_n, \phi_{(\x,
    x_n)}(\e)) - ax_n$.  It follows that there is a $\mathcal
  P_\pm$-preserving diffeomorphism of $\rr^{n-1} \times \rr^{1+N}
  \times \rr^{1+N}$ taking the sublevel sets $\Delta_a^\lambda$ to
  $(\Delta_a')^\lambda$, for all $\lambda$.  The naturality of the
  long exact sequences then implies that the capacities are unchanged.
  
  Lastly, suppose $F'$ is obtained from $F$ by stabilization, i.e.,
  $F'(\x, x_n, \mathbf{e}, \mathbf{e}') = F(\x, x_n, \mathbf e) +
  Q(\mathbf e')$, for some non-degenerate quadratic function $Q: \rr^q
  \to \rr$.  It then follows that the associated difference function
  $\Delta_a'$ is a stabilization of $\Delta_a$:
  $$\Delta_a'(\x, x_n, \mathbf e, \mathbf e', \tilde x_n, 
  \tilde{\mathbf e}, \tilde{\mathbf e}') = \Delta_a(\x, x_n, \e,
  \tilde x_n, \tilde \e) + Q(\mathbf e') - Q(\tilde{\mathbf e}').$$ In
  general, for any function $G: \rr^m \to \rr$ and any non-degenerate
  quadratic $Q$, we will show that for $a < b$, there is a natural
  isomorphism:
  \begin{equation} \label{eqn:hom-stab}
    H^* (G^b, G^a) \simeq H^{* + \ind Q}((G \oplus Q)^b, (G \oplus Q)^a).
  \end{equation}
  It suffices to consider $Q: \rr \to \rr$ with either $Q(e) = e^2$ or
  $Q(e) = -e^2$.  For the first case, it is easy to see that the pair
  $(G \oplus Q)^b, (G \oplus Q)^a)$ deformation retracts to $(G^b,
  G^a)$.  For the second case, the isomorphism in
  Equation~\eqref{eqn:hom-stab} can be seen as follows: after a
  homeomorphism and deformation retract, $((G \oplus Q)^b, (G \oplus
  Q)^a)$ becomes 
  \begin{multline} \label{eqn:hom-stab-2}
    \bigl((\rr^m \times (-\infty, -1] \cup [1, \infty)) \cup G^b
    \times [-1,1], \\ (\rr^m \times (-\infty, -1] \cup [1, \infty)) \cup
    G^a \times [-1,1] \bigr).
  \end{multline}
  By excision, the cohomology groups of the pair in
  Equation~\eqref{eqn:hom-stab-2} are isomorphic to: 
  \begin{equation}
    H^*\bigl(G^b \times [-1,1], (G^b \times \{\pm 1\}) \cup (G^a \times
    [-1,1]) \bigr),
  \end{equation}
  which can be regarded as the cohomology groups of a suspension of
  $(G^b, G^a)$, as desired.
\end{proof}

% **********
\section{Properties of the Capacities}
\label{sec:properties}

With the capacities now defined, we are ready to prove the five
properties of these capacities enumerated in
Theorem~\ref{thm:properties}.

% *****
\subsection{Monotonicity}
\label{ssec:mono}

Let $L_a, L_b$ be generic slices of a planar Lagrangian with $0
\notin [a, b]$. If $L_{[a,b]} = \cup_{t \in [a,b]} L_t$ and if $j_t: L_t \to
L_{[a,b]}$ is the inclusion map then the goal of monotonicity is to compare
the capacities of the cohomology classes $j_a^*(u) \in H^k(L_a)$ and
$j_b^*(u) \in H^k(L_b)$, for some $u \in H^k(L_{[a,b]})$.

The purpose of splitting the domain of the difference function
$\Delta_a$ into $\mathcal{P}_\pm$ was to obtain the containments
$\Delta_{b, +}^\lambda \subset \Delta_{a, +}^\lambda$ and $\Delta_{a,
  -}^\lambda \subset \Delta_{b, -}^\lambda$; see equations and
relations (\ref{eqn:split-sublevels}) --
(\ref{eqn:difference-inclusions}).  This is the key fact in the proof
of Monotonicity.  We begin with a lemma:

\begin{lem} 
  \label{lem:mono-comm-diag} 
  Consider $a < b$ with $0 \notin [a, b]$.  Then the following diagram
  (and its analogues in the cases of the negative lower capacity and
  the upper capacities) commutes:
  \begin{equation*} \xymatrix@R=5pt{ & H^k(L_a)
      \ar[rr]^-{\varphi^{\lambda}_{a,+}}
      & & H^{k+N+1}(\Delta^{\eta}_{a,+}, \Delta^{\lambda}_{a,+}) \ar[dd]^{i^*} \\
      H^k(L_{[a,b]}) \ar[ur]^{j^*_a} \ar[dr]_{j^*_b} & & &  \\
      & H^k(L_b) \ar[rr]_-{\varphi^{\lambda}_{b,+}}  & &
      H^{k+N+1}(\Delta^{\eta}_{b,+}, \Delta^{\lambda}_{b,+}).}
  \end{equation*}
\end{lem}

Assuming the lemma for now, we prove Monotonicity:

\begin{proof}[Proof of Monotonicity] To prove the inequality $c_+^{L,
    a} (j^*_a u) \leq c_+^{L, b}(j^*_b u)$, note that
  Lemma~\ref{lem:mono-comm-diag} shows that if
  $\varphi_{a,+}^\lambda(j^*_au) = 0$ then
  $\varphi_{b,+}^\lambda(j^*_bu) = 0$.  Thus, we have:
  \begin{align*}
    c_+^{L, a}(j^*_au) &= \sup \{ \lambda: 
    \varphi_{a,+}^\lambda( j^*_a u ) = 0\} \\ &\leq \sup \{ \lambda: 
    \varphi_{b,+}^\lambda(j^*_bu) = 0 \} \\ &= c_+^{L, b}(j^*_bu).
  \end{align*}
  When comparing $c_-^{L, a}(j^*_au)$ and $c_-^{L, b}(j^*_bu)$, the
  fact that $\Delta_{a,-} \subset \Delta_{b, -}$ causes the map $i^*$
  to reverse direction, and hence reverses the inequality.  The proofs
  for the upper capacities are analogous.
%	  \footnote{ For the
%	    upper positive capacity, an analogous diagram shows
%	    $\Phi_{b,+}^\Lambda(u_a) \neq 0 \implies \Phi_{a,+}^\Lambda(u_b)
%	    \neq 0$, and thus
%	  $$C_+^{L, a}(u_a) = \inf \{ \Lambda: \Phi_{a,+}^\Lambda(u_a) \neq 0 \}
%	  \leq \inf \{ \Lambda: \Phi_{b,+}^\Lambda(u_b) \neq 0 \} =
%   C_+^{L, b}(u_b).$$ Again, when comparing $C_-^{L, a}(u_a)$ and
%   $C_-^{L, b}(u_b)$, the inequality will be reversed since
%   $\Delta_{a,-} \subset \Delta_{b, -}$.}

  To show that the inequalities are strict if one of the capacities
  involved is nonzero, we work as in \cite[Lemma
  4.7]{viterbo:generating}: note that for all but finitely many $a$,
  we may assume that there exists an $\epsilon$ so that if $t \in
  (a-\epsilon, a+\epsilon)$ then $\Delta_t$ has $k$ distinct nonzero
  critical values $c_1(t), \ldots, c_k(t)$ coming from $k$
  non-degenerate critical points.  Using arguments as in the proof of
  Lemma~\ref{lem:crit-pts}, these critical values come from smooth
  paths of Morse critical points $q_1(t), \ldots, q_k(t)$, which, in
  turn, come from double points of $\pi(L_t)$.  We compute that:
  \begin{align*}
    c'_i(t) = \partial_t (\Delta_t(q_i(t)))  &= (\partial_t
    \Delta_t)(q_i(t)) + d\Delta_t(q_i(t)) \partial_t(q_i(t)) \\
    &=  (\partial_t  \Delta_t)(q_i(t)) \\
    &= x_n(t) - \tilde{x}_n(t).
  \end{align*}
  Since $L_t$ is embedded, the $x_n$-heights of the double points of
  $\pi(L_t)$ must be different, and hence $c'_i(t) \neq 0$, as desired.
\end{proof}

\begin{proof}[Proof of Lemma~\ref{lem:mono-comm-diag}]
  The scheme of the proof is to introduce an \textbf{extended
    difference function} to express the cobordism $L_{[a,b]}$ in terms of
  level sets.  The extended difference function is defined by
  \begin{equation*}
    \Delta_{ab}(t, \x, x_n, \e, \tilde{x}_n, \tilde{\e}) = 
    \Delta_{t}(\x, x_n, \e, \tilde{x}_n, \tilde{\e}), \quad t\in [a,b].
  \end{equation*}
  The function and
  its level sets have the following properties, as can be seen by
  direct computation and the techniques of Lemma~\ref{lem:crit-pts}:
  \begin{enumerate}
  \item The sublevel sets of $\Delta_{ab}$ satisfy $\Delta_{ab}^\lambda =
    \bigcup_{t \in [a,b]} \Delta^\lambda_t$ and interact with
    $\mathcal{P}_\pm$ in the same way;
  \item $\Delta_{ab}$ is Morse-Bott with a critical submanifold of
    index $N+1$ that can be identified with $L_{[a,b]}$ via the fiber
    critical set $\Sigma_F$.
  \end{enumerate}

  The same constructions as in Section~\ref{sec:defn-cap} then allow
  us to define a map
  \begin{equation*}
    \varphi^\lambda_{ab, +}: H^k(L_{[a,b]}) \to H^{k+N+1}(\Delta^\eta_{ab,+},
    \Delta^\lambda_{ab,+}).
  \end{equation*}
  Similar constructions yield maps $\varphi^\lambda_{ab,-}$ and
  $\Phi^\Lambda_{ab,\pm}$.
  
  To finish the proof, we need only note that the naturality of the
  Mayer-Vietoris sequence, the long exact sequence of a triple, and
  the Gysin sequence shows that the top and bottom parallelograms of
  the following diagram commute, as does the rightmost
  triangle: 
   \begin{equation*} \label{eqn:mono-diag} \xymatrix{ & H^k(L_a)
      \ar[rr]^-{\varphi^{\lambda}_{a,+}}
       & & H^{k+N+1}(\Delta^{\eta}_{a,+}, \Delta^{\lambda}_{a,+}) \ar[dd]^{i^*} \\
      H^k(L_{[a,b]}) \ar[ur]^{j^*_a} \ar[dr]_{j^*_b}
      \ar[rr]^-{\varphi^\lambda_{ab,+}} & & 
      H^{k+N+1}(\Delta^\eta_{ab,+}, \Delta^\lambda_{ab,+})
      \ar[ur]_{J^*_a} \ar[dr]^{J^*_b}
      &  \\
      & H^k(L_b) \ar[rr]_-{\varphi^{\lambda}_{b,+}}  & &
    H^{k+N+1}(\Delta^{\eta}_{b,+}, \Delta^{\lambda}_{b,+})}
  \end{equation*}
\end{proof}

\subsection{Continuity}

By fixing a cohomology class in $H^*(L_{[a,b]})$, not only can we
compare capacities at different levels, but we also create a
continuous, piecewise differentiable function from each capacity.

\begin{proof}[Proof of Continuity]
  Suppose that $a$ and $b$ are generic levels with $0 < a < b$ (the
  proof is entirely similar for negative levels).  Using a
  fiber-preserving diffeomorphism of $\rr^{n-1} \times \rr^{N+1}$ that
  sends $x_n$ to $\frac{a}{t}x_n$, we may modify the functions $F_t$,
  $t \in [a,b]$, to functions $\bar{F}_t$ that agree outside a compact
  set.  An application of the argument in Lemma~\ref{lem:cap-well-def}
  shows that these modified functions yield difference functions
  $\bar{\Delta}_t$ that give the same capacities as before.  It
  follows that for every $\epsilon>0$, there is a $\delta>0$ so that
  if $|s-t| < \delta$, then $\| \bar{\Delta}_s - \bar{\Delta}_t \| <
  \epsilon$. In particular, we obtain the inclusions:
  \begin{equation*}
    \bar{\Delta}_{s, \pm}^{\lambda - \epsilon} \subset
    \bar{\Delta}_{t,\pm}^\lambda \subset \bar{\Delta}_{s,\pm}^{\lambda+\epsilon}.
  \end{equation*}

  These inclusions lead to the following commutative diagram for the
  lower capacities:
  \begin{equation*} \label{eqn:mono} \xymatrix{ & &
      H^{k+N+1}(\bar\Delta^{\eta}_{s,\pm},
      \bar\Delta^{\lambda-\epsilon}_{s,\pm}) \\
      H^k(L) \ar[rr]^-{\varphi^\lambda_{s,\pm} \circ j^*_s}
      \ar[urr]^-{\varphi^{\lambda - \epsilon}_{t, \pm} \circ j^*_t}
      \ar[drr]_-{\varphi^{\lambda+\epsilon}_{s, \pm} \circ j^*_s} & &
      H^{k+N+1}(\bar\Delta^{\eta}_{t,\pm}, \bar\Delta^{\lambda}_{t,\pm})
      \ar[u]_{i^*} \\
      & & H^{k+N+1}(\bar\Delta^{\eta}_{s,\pm},
      \bar\Delta^{\lambda+\epsilon}_{s,\pm}) \ar[u]_{i^*}.  }
  \end{equation*} It follows immediately from the bottom and top
  triangles, respectively, that:
  \begin{equation*} c(s) - \epsilon \leq
    c(t) \leq c(s) + \epsilon,
  \end{equation*} 
  as required.  The proof for the upper capacities is similar.

  The fact that $c(t)$ is monotone shows that its one-sided limits
  exist at a non-generic level $t$, and the proof above with $a =
  t-\epsilon$ and $b = t+\epsilon$ shows that these two limits are
  equal. Thus, the non-generic levels are removable singularities.

  Finally, piecewise differentiability follows from the arguments at
  the end of the proof of Monotonicity in the previous section.
\end{proof}

\subsection{Invariance}

The capacities for a slice $L_a$ depend on the entire Lagrangian $L$
since they are defined in terms of a generating family for $L$.  In
this section, however, we will show that isotoping the Lagrangian
while keeping the slice $L_a$ unchanged will not change the
capacities.  Specifically, let $\psi_t$, $t \in [0,1]$, be a compactly
supported symplectic isotopy of $\rr^{2n}$.  Let $L(t) = \psi_t(L)$ be
the image of $L$ under the isotopy, and suppose that the slice $L_a(t)
= L(t) \cap \{y_n = a\}$ is always equal to $L_a$.

The proof of Invariance, namely that the capacities of $L_a$ do not
depend on $L(t)$, follows easily from the fact that the capacities lie
in the discrete set of the critical values of the difference function
(see Lemma~\ref{lem:cap-at-crit}) and from following proposition about
the continuity of the capacities, whose proof is a slight variation on
the proof of continuity given above:

\begin{prop} \label{prop:cont} With the notation above, for $u \in
  H^k(L_a)$, the capacities $c_\pm^{L(t),a}(u)$ and
  $C_\pm^{L(t),a}(u)$ are continuous in $t$.
\end{prop}

\begin{proof} As noted in Remark~\ref{rem:gf-deform}, there is a
  $1$-parameter family $F(t)$ of quadratic-at-infinity generating
  families for $L(t)$.  Now use the same proof as in the previous
  section.
\end{proof}

\subsection{Non-Vanishing}

Next, we will prove Non-Vanishing, which asserts that, for $u \neq 0$,
at least one of the four capacities $c_\pm^{L,a}(u)$, $C_\pm^{L,a}(u)$
is nonzero.  The scheme of the proof is to understand the cohomology
of $(\Delta_a^\theta, \Delta_a^{-\theta})$ for large $\theta$, and
then to use this to relate the capacities.

\begin{lem} \label{lem:trivial-hom} For $\theta \gg 0$,
  $H^*(\Delta_a^\theta, \Delta_a^{-\theta}) = 0$.
\end{lem}

A component in the proof of this lemma is the following ``critical
non-crossing'' lemma, whose proof can be found, for example, almost
word-for-word in the proof of Lemma 3.10 of \cite{lisa:links}.
 
\begin{lem} \label{lem:loc-1-param} Consider a smooth $1$-parameter
  family of difference functions $\Delta_a(t)$ that arises from a
  $1$-parameter family of quadratic-at-infinity generating families
  $F(t)$ and levels $a(t)$.  Suppose $\tau, \sigma: [0,1] \to \rr$ are
  continuous paths such that $\tau(t)$ and $\sigma(t)$ are regular
  values of $\Delta_a(t)$ with $\tau(t) < \sigma(t)$ for all $t$.
  Then for any $s, t \in [0,1]$, there is an isomorphism
  \begin{equation*} H^*\left( \Delta_a(s)^{\sigma(s)},
      \Delta_a(s)^{\tau(s)} \right) \simeq H^*\left(
      \Delta_a(t)^{\sigma(t)}, \Delta_a(t)^{\tau(t)} \right).
  \end{equation*}
\end{lem}

\begin{proof}[Proof of Lemma~\ref{lem:trivial-hom}] Suppose $a>0$; the
  proof for $a<0$ is entirely similar. Choose $\theta$ large enough
  so that the magnitudes of the critical values of $\Delta_b$ are less
  than $\theta$, for all $b > a$.  Choose $b>a$ such that the slice
  $L_b$ is empty; this is possible since $L$ is flat-at-infinity.
  Since $L_b$ is empty, Lemma~\ref{lem:crit-pts} shows that the
  difference function $\Delta_b$ has no critical points, and hence
  that $H^*(\Delta_b^\theta, \Delta_b^{-\theta}) = 0$.  We can then
  conclude, using Lemma~\ref{lem:loc-1-param} with the constant paths
  $\pm \theta$, that:
  \begin{equation*} H^*(\Delta_a^\theta, \Delta_a^{-\theta}) \simeq
    H^*(\Delta_b^\theta, \Delta_b^{-\theta}) = 0.
  \end{equation*}
\end{proof}

We now leverage Lemma~\ref{lem:trivial-hom} to relate the spaces
involved in the definition of the capacities.

\begin{lem} \label{lem:exactness} For $\theta \gg 0$,
  $\ker p_a^{-\theta} = \ker D_a^\theta$.
\end{lem}

\begin{proof} 
  The connecting homomorphisms in the exact sequences of the triples
  $(\Delta_a^\theta, \Delta_a^{\eta}, \Delta_a^{-\eta})$ and
  $(\Delta_a^\theta, \Delta_a^{\eta}, \Delta_a^{-\theta})$ together
  with the map $p^{-\theta}_a$ in the exact sequence of the triple
  $(\Delta_a^\eta, \Delta_a^{-\eta}, \Delta_a^{-\theta})$ fit
  together in the following commutative diagram:
  \begin{equation*}
    \xymatrix@!R=4pt{
      & H^{k+N+1}(\Delta_a^\eta, \Delta_a^{-\theta}) \ar[dd]^{d^*} \\ 
      H^{k+N+1}(\Delta_a^\eta, \Delta_a^{-\eta}) \ar[ur]^{p^{-\theta}_a}
      \ar[dr]_{D_a^\theta} & \\
      & H^{k+N+2}(\Delta_a^\theta, \Delta_a^{\eta}).
    }
  \end{equation*}
  By Lemma~\ref{lem:trivial-hom}, the connecting homomorphism $d^*$ at
  the right side of the diagram is an isomorphism, so $\ker
  D_a^\theta = \ker d^* \circ p^{-\theta}_a = \ker p^{-\theta}_a$,
  as desired.
\end{proof}

\begin{proof}[Proof of Non-Vanishing] Suppose that $0
  \neq u \in H^k(L_a)$ and $c_+^{L,a}(u) = c_-^{L,a}(u) = 0$.  Then,
  by Corollary \ref{cor:vanishing_capacities}, for all $\lambda$, $u
  \notin \ker p_{a}^\lambda$.  In particular, for $\theta \gg 0$,
  Lemma~\ref{lem:exactness} implies that $u \notin \ker
  p_{a}^{-\theta} = \ker D_a^\theta$.  So, again by Corollary
  \ref{cor:vanishing_capacities}, either $C_+^{L,a}(u) \neq 0$ or
  $C_-^{L,a}(u) \neq 0$.
\end{proof}

\subsection{Conformality}

The proof of conformality follows immediately from two facts:
\begin{enumerate}
\item If $F$ generates $L$ then
  $$\bar{F}(\x, x_n, \e) = \beta^2 F(\frac{1}{\beta}\x, \frac{1}{\beta}
  x_n, \e)$$ generates $\beta L$.
\item The following sublevel sets are diffeomorphic: 
  $$\Delta_a^\lambda \simeq \bar{\Delta}_{\beta a}^{\beta^2 \lambda}.$$
\end{enumerate}

% **********
\section{Applications and Extensions}
\label{sec:apps}

% *****
\subsection{Obstructions to Cobordisms}
\label{ssec:obstructions}

In order to prove Theorems~\ref{thm:8cobord} and
\ref{thm:higher-dim8s}, we begin with two useful lemmas. The first is
a gluing result:

\begin{lem}[\cite{arnold-paper}]
  \label{lem:gluing}
  Let $L, L' \subset \rr^{2n}$ be two Lagrangians that are transverse
  to and agree on $\{ y_n = a \}$.
  For all $\epsilon>0$, there exists a Lagrangian $L''$ such that:
  \begin{enumerate}
  \item $L'' \cap \{y_n < a - \epsilon \} = L \cap \{y_n < a -
    \epsilon \}$ and 
  \item $L'' \cap \{y_n > a + \epsilon \} = L' \cap \{y_n > a +
    \epsilon \}$.
  \end{enumerate}
\end{lem}

\noindent
Note that the proposition in \cite{arnold-paper} is stated for
Lagrangians in $\rr^4$, but the proof applies to higher dimensions.
In addition, in \cite{arnold-paper}, the Lagrangians are assumed to be
flat-at-infinity and planar, but the proof can be carried out near
$L_a = L'_a$ in a standard neighborhood of either one of the
Lagrangians.

The second is a local computation of critical values of $\Delta_a$
from the geometry of $\pi(L_a)$:

\begin{lem} \label{lem:crit-val} Suppose that $q=(\x, x_n, \e,
  \tilde{x}_n, \tilde{\e})$ is a non-degenerate critical point of
  $\Delta_a$ and that there exists a path $\gamma_q: [0,1] \to
  \Sigma_{F_a}$ that begins at $(\x, \tilde{x}_n, \tilde{\e})$ and
  ends at $(\x, x_n, \e)$.  Let $i_a: \Sigma_{F_a} \to \pi(L_a)$
  denote the inclusion as in equation (\ref{eqn: crit_inclusion}).
  Then the critical value of $(\x, x_n, \e, \tilde{x}_n,
  \tilde{\e})$ is:
  \begin{equation*} \label{eqn:delta-area} \Delta_a(\x, x_n,
    \mathbf{e}, \tilde{x}_n, \tilde{\mathbf{e}}) = \int_{ i_a
      \circ \gamma_q} \y \, d\x.
  \end{equation*}
\end{lem}

\begin{proof} A direct computation using equation (\ref{eqn:
    crit_inclusion}) shows that $i_a^* (\y \, d\x) = dF_a$, and hence
  we have:
  \begin{equation*} \label{eqn:delta-FTC} \Delta_a(\x, x_n,
    \e, \tilde{x}_n, \tilde{e}) = \int_{\gamma_q} dF_a =
    \int_{i_a \circ \gamma_q} \y \, d\x.
  \end{equation*}
\end{proof}

Note that by Stokes' Theorem, this critical value can be thought of as
the negative symplectic area of a $2$-chain bounded by the loop $i_a
\circ \gamma_q$.  In particular, the critical values of $\Delta_a$ are
determined by the slice $L_a$.

\begin{exam} \label{exam:diagram-crit-vals} In Example
  \ref{exam:explicit_crit_values}, we calculated the critical values
  of the non-degenerate critical points of an explicit generating
  family for a Lagrangian with slices agreeing $8^1_-(r)$.  Using
  Lemma~\ref{lem:crit-pts}, we see that for {\it any} difference
  function $\Delta_a$ associated to {\it any} Lagrangian having a
  slice $L_a$ agreeing with $8^1_-(r)$, the non-degenerate critical
  points come from the double point of $\proj(L_a)$: one with $x_2 >
  \tilde x_2$ and another with $ x_2 < \tilde x_2$.  The loop $i_a
  \circ \gamma_q$ of Lemma~\ref{lem:crit-val} associated to the point
  in $\mathcal P_-$ (where $x_2 > \tilde x_2$) will be oriented
  clockwise, and thus this critical point has critical value $v > 0$
  equal to the area of one of the lobes of our projected figure-8
  curve.  The loop $i_a \circ \gamma_q$ associated to the critical
  point in $\mathcal P_+$ will be the same curve (but now oriented
  counterclockwise), and so the critical point will have critical
  value $-v < 0$.  $\diamond$
\end{exam}

\begin{proof}[Proof of Theorem~\ref{thm:higher-dim8s}] 
  First notice that, for any $s > 0$, if there is an unknotted planar
  Lagrangian $L$ so that $L_a = i_a(8_-^{n-1}(s))$, then $a > 0$;
  otherwise, it would be possible to construct an (exact) embedded
  Lagrangian sphere as follows.  If $a < 0$, then there exists a
  Lagrangian disk in $\{ y_n \leq a\}$ with boundary
  $i_a(8_-^{n-1}(s))$.  By the construction in
  Example~\ref{exam:cylinder_construct}, for some $a'>0$, there exists
  a Lagrangian disk $\{ y_n \geq a' \}$ with boundary
  $i_{a'}(8_-^{n-1}(s))$.  Using Lemma~\ref{lem:gluing}, translations
  of these disks can be glued to form an embedded (exact) Lagrangian
  sphere, which is impossible for $n > 1$ by a result of Gromov
  \cite{gromov:hol}.
 
  Next notice that, for any $s > 0$, if there is an unknotted planar
  Lagrangian $L$ so that $L_a = i_a(8_-^{n-1}(s))$, then $C_+^{L,a}(u)
  = 0$ and $c_-^{L,a}(u) = 0$, for all $u \in H^k(L_a)$. This follows
  from Monotonicity and the facts that $a > 0$ and that for $\tau \gg
  0$, we have $L_\tau = \emptyset$.

  In addition, notice that if there exists an unknotted planar
  Lagrangian $L \subset \rr^{2n}$ so that $L_a = i_a(8_-^{n-1}(s))$,
  then the critical values of $\Delta_a$ are $0, \pm v_s$ where $v_s >
  0$ and does not depend on the height $a$.  The fact that there are
  three critical values of $\Delta_a$ follows from
  Lemma~\ref{lem:crit-pts} since $\pi(L_a)$ has a single double point;
  if all critical values were $0$ then there would be a contradiction
  to Non-Vanishing.  The fact that the critical value $v_s > 0$ does
  not depend on the height $a$ follows from Lemma~\ref{lem:crit-val}.

  For fixed $r < R$, let $v_r$ and $v_R$ be the positive critical
  values associated to the difference functions for heights giving
  slices $8_-^{n-1}(r)$ and $8_-^{n-1}(R)$, respectively.  We can then
  deduce that $v_r < v_R$ as follows.  By the construction in
  Example~\ref{exam:cylinder_construct}, there exists an unknotted
  planar Lagrangian $L\subset \rr^{2n}$ and $a < b$ so that
  $$L_a = i_{a}(8_-^{n-1}(R)), \qquad L_{b} = i_{b}(8_-^{n-1}(r)).$$
  Non-vanishing implies that for all non-trivial $u \in
  H^k(L_{[a,b]})$, $C_-(j_b^*u) = v_r$ or $c_+(j_b^*u) = -v_r$.  If
  $C_-(j_b^*u) = v_r$, then by Monotonicity, $C_-(j_a^*u) = v_R >
  v_r$.  Similarly, $c_+(j_b^*u) = -v_r$ implies, by Monotonicity, $
  c_+(j_a^*u) = -v_R < -v_r$, again implying $v_R > v_r$.

  The desired result now follows by contradiction: suppose there
  exists an unknotted planar Lagrangian $L \subset \rr^{2n}$ and $a <
  b$ so that
  $$L_a = i_a(8_-^{n-1}(r)), \qquad L_b =  i_b(8_-^{n-1}(R)), \qquad r \leq R.$$
  Repeating the Non-vanishing and Monotonicity arguments as above, we
  find that $v_R < v_r$.  This shows that $r = R$ is not possible and
  gives us a contradiction to the above paragraph when $r < R$.
\end{proof}

\begin{rem} The above proof uses in an essential way that there are
  just two non-zero critical values of $\Delta_a$.  If, for example,
  it were true that every flat-at-infinity planar Lagrangian in
  $\rr^{2n}$ is unknotted, then the proof of anti-symmetry in
  \cite{arnold-paper} would extend to higher dimensional cobordisms
  between many different pairs of submanifolds.
\end{rem}

To prove Theorem~\ref{thm:8cobord}, which has no extendibility
hypothesis on the cobordism, we need the following geometric result,
whose proof is a simple modification of
\cite[Theorem 6.4]{chantraine} with the writhe taking the place of the
Thurston-Bennequin number. 

\begin{lem}
  \label{lem:writhe}
  If $L$ be an embedded Lagrange cobordism between the slices $L_a$
  and $L_b$ with $a < b$, then
  $$\chi(L) = \writhe(L_b) - \writhe(L_a),$$
  where the writhe $\writhe$ is calculated with respect to the blackboard
  framing of the $x_1y_1$ projection.
\end{lem}

\begin{proof}[Proof of Theorem~\ref{thm:8cobord}]  Suppose that there
exists a Lagrangian cobordism $L \subset \{ a \leq y_2 \leq b \} \subset \rr^4$
with $L_a = i_a(8_-^1(r))$ and $L_b = i_b(8_-^1(R))$.  This
cobordism can be extended to a flat-at-infinity planar Lagrangian as
follows.  After perhaps translating $L$, we know, by Example~\ref{exam:cylinder_construct}, that 
there exists a Lagrangian disk in $\{ y_2 \geq b\}$ with boundary
$ i_b(8_-^1(R))$ and a flat-at-infinity planar Lagrangian with a disk removed that lies
  in $\{y_2 \leq a\}$ and has boundary $i_a(8^1_-(r))$.  Using 
  Lemma~\ref{lem:gluing}, we can glue the cobordism, the disk,
  and the punctured plane together to form a flat-at-infinity Lagrangian.     
   Since $\writhe(8^1_-(r)) =
  \writhe(8^1_-(R))$, Lemma~\ref{lem:writhe} shows that the Lagrange
  cobordism $L$ must be topologically an annulus.  Thus, the
  Lagrangian $L'$ must be planar.  By a result of Eliashberg and Polterovich
  \cite{ep:local-knots}, $L'$ is unknotted, and we can apply the proof
  of Theorem~\ref{thm:higher-dim8s} to show that no such $L'$ exists.
\end{proof}

Remark~\ref{rem:alt-8cobord}, below, gives an alternate proof of
Theorem~\ref{thm:8cobord} using explicit calculations of capacities.

% **********
\subsection{$4$-dimensional Computations}
\label{ssec:compute}

We have already seen that the functorial properties of
the capacities can be used to obtain obstructions to cobordisms.  
In order to get some understanding of what the capacities are
measuring, we now explicitly calculate  capacities of some slices.
Recall that in general, capacities for a slice of a Lagrangian
depend on the entire Lagrangian.  However, we will show that
when the slice is particularly simple, for example when the
slice agrees with $8^1_\pm(r)$, the capacities only depend
on the slice.  

A key step to doing explicit calculations is to determine the critical
 values and indices of
critical points of $\Delta_a$ geometrically from the projection
$\pi(L_a)$.  The computation of critical values follows from
Lemma~\ref{lem:crit-val} in the previous section where it is shown that the  
critical value associated to a critical point $q=(x_1, x_2, \e,
\tilde{x}_2, \tilde{\e})$ of $\Delta_a$ is 
 obtained by
 integrating along a path $i_a \circ \gamma_q \subset \pi(L_a)$.  Similarly,
 the index of an isolated critical point $q$ of $\Delta_a$ may also be calculated using
the path $\gamma_q$.  The path $\gamma_q$ defines a path $\Gamma_q$ of
Lagrangian subspaces of $\rr^2$ via
$$\Gamma_q(t) = T_{i_a \circ \gamma_q(t)} \proj(L_a).$$
Close this path to a loop $\bar{\Gamma}_q$ by rotating $\Gamma_q(1)$
clockwise until it coincides with $\Gamma_q(0)$.  The index of
$\Delta_a$ at $q$ can then be calculated in terms of the Maslov index
of $\bar{\Gamma}_q$:

\begin{lem}
  \label{lem:morse-index} If $q$ is an isolated critical point of
  $\Delta_a$ and $\gamma_q$ is a path in $\rr \times \rr^{1+N}$ as in
  Lemma~\ref{lem:crit-val}, then:
  \begin{equation*} \ind_q \Delta_a = -\mu(\bar{\Gamma}_q) + (N+1).
  \end{equation*}  
\end{lem}

\begin{proof} Equation~(\ref{eqn:hessian-delta}) shows that:
  \begin{equation} \label{eqn:hessian-ind} \ind_q \Delta_a =
    \ind_{x_1} (g''-\tilde{g}'') + (\ind A - \ind \tilde{A}) + (N+1).
  \end{equation} 
  To interpret this formula in terms of the Maslov index, consider
  that the path of Hessians $d^2F_a(\gamma_q(s))$ generates the path
  of Lagrangian subspaces $\Gamma_q$ in the language of Th\'eret
  \cite[Appendix B]{theret:camel}.  
  Th\'eret further shows
  that 
  \begin{equation*} \label{eqn:maslov-ind} \mu(\Gamma_q) =
    \ind_{\gamma_q(0)} F_a - \ind_{\gamma_q(1)} F_a.
  \end{equation*} 
  (Note that Th\'eret takes the opposite sign convention for the
  Maslov index than the one used here and in
  \cite{robbin-salamon:maslov}.)  By construction, we may compute that
  \begin{align*} d^2F_a(\gamma_q(0)) &= \begin{bmatrix} \tilde{g}''
      & 0 \\ 0 & \tilde{A}
    \end{bmatrix}, & d^2F_a(\gamma_q(1)) &= \begin{bmatrix} g'' & 0
      \\ 0 & A
    \end{bmatrix}.
  \end{align*} 
  Combining this with Equations (\ref{eqn:hessian-ind}) and
  (\ref{eqn:maslov-ind}), we obtain
  \begin{multline*} \ind_q \Delta_a = (N+1) -\mu(\Gamma_q) + \left[ \ind
      (\tilde{g}'' - g'') - \ind (\tilde{g}'') + \ind (g'') \right].
  \end{multline*} It is easy to check that the term $\ind (\tilde{g}''
  - g'') - \ind (\tilde{g}'') + \ind (g'')$ accounts precisely for the
  change in Maslov index engendered by closing $\Gamma_q$ into
  $\bar{\Gamma_q}$. This completes the proof of the index formula for
  an isolated critical point.
\end{proof}

\begin{exam} In Example \ref{exam:explicit_crit_values}, we calculated
  the indices of the non-degenerate critical points of an explicit
  generating family (with $N = 0$) of a Lagrangian with slices
  agreeing with $8^1_-(r)$.  Lemma~\ref{lem:morse-index} gives us a
  way to calculate the indices of the non-degenerate critical points
  of {\it any} difference function $\Delta_a$ of {\it any} Lagrangian
  having a slice agreeing with such a unknotted figure-8 curve with a
  negative crossing. Notice that the loop $\bar \Gamma_{q_-}$ associated to
  the critical point $q_- \in \mathcal P_-$ ($x_2(q_-) > \tilde
  x_2(q_-)$) will make one full rotation in the clockwise direction,
  and thus $\mu(\bar\Gamma_{q_-}) = -2$.  So for any difference function
  $\Delta_a$ for $L_a$ with domain $\rr \times \rr^{1+N} \times
  \rr^{1+N}$, $\ind_{q_-}\Delta_a = 2 + N+1 = N+ 3$.  The loop $\bar
  \Gamma_{q_+}$ associated to $q_+ \in \mathcal P_+$ will make a half
  rotation in the counterclockwise direction, and so
  $\ind_{q_+}\Delta_a = -1 + N + 1 = N$. $\diamond$
\end{exam}

Now that we can calculate critical values and indices using only
$L_a$, let us proceed to calculate the capacities of $8_\pm^1(r)$
or, more generally, any slice with a projection diagram
 agreeing with that of 
$8_\pm^1(r)$.  
   To set up the geometric
situation more precisely, let $L_a$ and $L_a'$ be slices that project
to immersed curves in the $(x_1, y_1)$-plane with one double point
$q$.  Let $q^+, q^-$ denote the preimages of $q$ in $L_a$ (resp.\
$L_a'$) with $x_2(q^+) > x_2(q^-)$, and assume that the crossing in
$L_a$ (resp.\ $L_a'$) is negative (resp. positive) so that the path
$i_a \circ \gamma_q$ from $q^+$ to $q^-$ in $L_a$ (resp.\ $L_a'$)
projects to a loop traversed in the counterclockwise (resp.\
clockwise) direction.  Let $A$ denote the absolute value of the area
of the region bounded by $\pi \circ i_a \circ \gamma_q$. With this
notation, we have the following calculation of capacities:

\begin{prop}  \label{prop:4-dim-compute} Let $L, L'$ be any flat-at-infinity planar Lagrangians
  with slices $L_a, L_a'$ as described above.  Then, for $0 \neq u \in
  H^0(L_a)$ and $0 \neq v \in H^1(L_a)$,
  $$ 
  c_+^{L,a} (u) = -A, \quad c_-^{L,a}(u) = 0, \quad C_+^{L,a}(v) = 0,
  \quad C_-^{L,a}(v) = A.
  $$
  For $0 \neq u \in H^0(L'_a)$ and $0 \neq v \in H^1(L'_a)$,
  $$ 
  c_+^{L',a} (u) =0, \quad c_-^{L',a}(u) = -A, \quad C_+^{L',a}(v) = A,
  \quad C_-^{L',a}(v) = 0.
  $$
\end{prop}

\begin{proof} Let $F: \rr^2 \times \rr^N \to \rr$ be a
  quadratic-at-infinity generating family for any Lagrangian $L$ whose
  slice at $y_2 = a$ is $L_a$.  By Lemma \ref{lem:crit-pts}, $F_a$
  must have two non-degenerate critical points and a non-degenerate
  critical submanifold.  By Lemma \ref{lem:crit-val}, the loop $i_a
  \circ \gamma_q$ with clockwise orientation is associated with the
  critical point with positive critical value, and hence this critical
  point lies in $\mathcal P_- = \{ x_2 > \tilde x_2 \}$.  By
  Lemmas~\ref{lem:crit-val} and \ref{lem:morse-index}, the critical
  value of this point is $A$, and its index must be $N+3$.  By a
  similar argument, the other critical point occurs in $\mathcal P_+=
  \{ x_2 < \tilde x_2 \}$, has critical value $-A$ and index $N$.
  Calculations of the non-degenerate critical points, values, and
  indices for $L_a'$ are analogous.

  We will compute the lower capacities $c_\pm^{L,a} (u)$; calculations
  for the upper capacities $c_\pm^{L,a} (v)$ and all capacities of
  $L'_a$ follow analogously. To show that $c_-^{L,a}(u) = 0$, notice
  that the index and critical value calculations above, combined with
  Lemma~\ref{lem:def-retr}, show that $H^k(\Delta_{a, -}^{-\eta},
  \Delta_{a, -}^\lambda ) = 0$, for all $k$ and for any $\lambda
  <-\eta$.  Thus, examining the exact sequence of the triple
  $(\Delta_{a,-}^\eta, \Delta_{a, -}^{-\eta}, \Delta_{a,
    -}^{\lambda})$, we see that $\varphi^{\lambda}_{a,-}$ is an
  isomorphism for all $\lambda$, and so we are done by
  Corollary~\ref{cor:vanishing_capacities}.  Similar arguments using
  the fact that there are no critical points of index $N+2$ with
  positive critical value show that the upper capacities
  $C_\pm^{L,a}(u)$ also vanish.  Thus, by Non-Vanishing, it must be
  the case that $c_+^{L,a}(u) \neq 0$. Since $-A$ is the only negative
  critical value of $\Delta_a$, Lemma~\ref{lem:cap-at-crit} shows that
  $c_+^{L,a}(u) = -A$, as desired.
\end{proof}

\begin{rem}  \label{rem:alt-8cobord}
An alternate proof to Theorem~\ref{thm:8cobord} follows
from gluing arguments,
 this explicit calculation of the capacities, and Monotonicity.
\end{rem}

% *****
\subsection{Non-Squeezing Phenomena}
\label{ssec:squeeze}

As in the proofs of the non-existence results, the non-squeezing
results in Theorems~\ref{thm:8squeeze} and \ref{thm:ht-cap} rely
primarily on Monotonicity.

\begin{proof}[Proof of Theorem~\ref{thm:ht-cap}]
  We prove the inequality for $C_-$; the other proofs are entirely
  analogous.  Fix $u \in H^*(L)$.  Recall from
  Proposition~\ref{prop:cont} that $C_-(t) = C_-^{L,t}(j^*_tu)$ is a
  continuous, piecewise differentiable function of $t$.  If the
  capacity comes from a piecewise continuous path of crossings
  $(\x(t)$, $x_n(t)$, $\tilde{x}_n(t))$, then the proof of
  Monotonicity shows that at all but finitely many levels $t$, $C_-'(t)
  = x_n(t) - \tilde{x}_n(t)$.  Since the capacity $C_-(t)$ comes from
  a critical point in $\mathcal{P}_-$ and, by hypothesis, $x_n(t)$ and
  $\tilde x_n(t)$ lie in an interval of length $\ell$, we have, for
  all but finitely many $t \geq a$:
  \begin{equation}
    -\ell \leq C_-'(t) \leq 0.
  \end{equation}
  The theorem now follows from this bound on the derivative and the
  hypothesis that when $t > w+a$, $C_-(t) = 0$ (since $L_t =
  \emptyset$).
\end{proof}  

\begin{proof}[Proof of Theorem~\ref{thm:8squeeze}] For any $r > 0$, a
  Lagrangian disk $L \subset \rr^4$ with boundary $\partial L = (L
  \cap \{y_2=a\} ) = i_a(8^1_\pm(r))$ can be extended to a
  flat-at-infinity planar Lagrangian.  This follows from the
  construction in Example~\ref{exam:cylinder_construct} and a gluing
  argument of Lemma~\ref{lem:gluing}.  The result then follows
  immediately as a corollary of Theorem~\ref{thm:ht-cap} given the
  calculation of the capacities in
  Proposition~\ref{prop:4-dim-compute}.  \end{proof}

% *****
\subsection{Capacities and Field Theory}
\label{ssec:ft}

To bring the capacities into an SFT-style framework, we use the
relative cohomology groups of sublevel sets to assign a filtered
cohomology theory --- really, four filtered cohomology theories --- to
each slice.  To a Lagrange cobordism between slices, we assign a
filtered homomorphism for each of the filtered cohomologies.  As we
shall see, the direction of the homomorphism depends on whether we
consider sublevel sets in $\mathcal{P}_+$ or in $\mathcal{P}_-$.
Further, the capacities can be used to detect non-triviality of the
homomorphisms.

Based on the definitions of the capacities, we define:

\begin{defn}
  \label{defn:homology}
  The \textbf{positive and negative lower filtered cohomology groups}
  of a generic slice of an unknotted planar Lagrangian $L$ are
  $$h^*_{\pm, \lambda}(L,a) = H^{*+N+1}(\Delta_{a,\pm}^{-\eta},
  \Delta_{a,\pm}^\lambda),$$ where $\lambda < 0$ is the filtration
  level and $\eta$ is any positive real number such that there are no
  critical values of $\Delta_a$ in $[-\eta, 0)$.  The \textbf{positive
    and negative upper filtered cohomology groups} are defined
  similarly:
  $$H^*_{\pm, \Lambda}(L,a) = H^{*+N+2}(\Delta_{a,\pm}^{\Lambda},
  \Delta_{a,\pm}^\eta),$$ for all $\Lambda > 0$.
\end{defn}

The proof of Lemma~\ref{lem:cap-well-def} shows that these
cohomologies are independent of the generating family used to define
$L$, up to an overall shift in degree.  Further,
Lemma~\ref{lem:def-retr} shows that the filtered cohomologies can only
change values when the filtration level passes through a critical
value of $\Delta_a$.

For $a < b$, the inclusions $i_+: \Delta^\lambda_{b,+} \hookrightarrow
\Delta^\lambda_{a,+}$ and $i_-: \Delta^\lambda_{a,-} \hookrightarrow
\Delta^\lambda_{b,-}$ from Equation~(\ref{eqn:difference-inclusions})
yield filtered homomorphisms between the cohomology groups of
different slices:
\begin{align*}
  i_+^*&: h^*_{+,\lambda}(L,a) \to h^*_{+,\lambda}(L,b), \quad
  i_+^*: H^*_{+,\Lambda}(L,a) \to H^*_{+,\Lambda}(L,b), \\
  i_-^*&: h^*_{-,\lambda}(L,b) \to h^*_{-,\lambda}(L,a), \quad i_-^*:
  H^*_{-,\Lambda}(L,b) \to H^*_{-,\Lambda}(L,a).
\end{align*}
As mentioned above, the capacities can give information about when
these maps are nontrivial.  For example, given $0 \not \in [a,b]$, the
analogue of the diagram in Lemma~\ref{lem:mono-comm-diag} for the
upper capacities immediately implies:

\begin{prop} \label{prop:cap-nontriv-homom} With notation as in
  Section~\ref{ssec:mono}, suppose that $C^{L,b}_+(j^*_b u) > 0$, for
  some $u \in H^k(W)$.  Then the map $i_+^*: H^k_{+,\Lambda}(L,a) \to
  H^k_{+,\Lambda}(L,b)$ is nontrivial for $\Lambda = C^{L,b}_+(j^*_b
  u)$.
\end{prop}

\begin{rem} \label{rem:leg-cob} For large enough $\Lambda$, the group
  $H^*_{+,\Lambda}(L,a) \oplus H^*_{-,\Lambda}(L,a)$ is equal to the
  generating family (co)homology of the lift of $\pi(L_a)$ to a
  Legendrian link in the standard contact $\rr^3 = J^1\rr$; see Fuchs
  and Rutherford \cite{fuchs-rutherford}, Traynor \cite{lisa:links},
  or Jordan and Traynor \cite{lisa-jill} for more information on
  generating family homology for Legendrian knots and links.  In fact,
  since $L$ is exact, it lifts to a Legendrian cobordism between the
  lifts of $\pi(L_a)$ and $\pi(L_b)$.  Thus, the field theory in this
  section suggests a structure for a field theory for Legendrian
  cobordisms between Legendrian knots defined by generating
  families. As Fuchs and Rutherford proved that the generating family
  cohomology of a Legendrian knot is isomorphic to the linearized
  contact homology (when these are defined), this remark justifies
  saying that the field theory defined above fits into an SFT
  framework.
\end{rem}

% *****

\bibliographystyle{amsplain} 
\bibliography{main}

\end{document}